\documentclass[reqno]{amsart}
\usepackage[english]{babel}
\usepackage{amsmath,longtable,amsfonts,amssymb,amsthm,amscd,latexsym,mathtools}
\usepackage[T1]{fontenc}
\usepackage[utf8]{inputenc}

\usepackage{tikz}
\usepackage{tikz-cd}
\usetikzlibrary{arrows}
\usepackage{color}
\usepackage{cite}
\usepackage{leftidx}
\usepackage[all]{xy}
\usepackage{marginnote}
\usepackage{ulem}
\usepackage{caption}
\usepackage{multirow}
\usepackage{hyperref}

\DeclareMathOperator{\Aut}{Aut}

\DeclareMathOperator{\id}{id}

\DeclareMathOperator{\lsp}{span}

\newtheorem{thm}{Theorem}[section]
\newtheorem{cor}[thm]{Corollary}
\newtheorem{lm}[thm]{Lemma}
\newtheorem{prop}[thm]{Proposition}
\newtheorem{defn}[thm]{Definition}
\newtheorem{rem}[thm]{Remark}
\newtheorem{exam}[thm]{Example}

\newtheorem{conj}[thm]{Conjecture}
\numberwithin{equation}{section}

\newcommand{\Ima}{\textup{Im}\:}

\newcommand{\F}{{\mathcal F}}
\newcommand{\G}{{\mathcal G}}
\newcommand{\N}{{\mathbb N}}
\newcommand{\Za}{{\mathrm Z}}

\newcommand{\A}{\mathbf{A}}
\newcommand{\Ba}{\mathbf{B}}
\newcommand{\zd}{C_2}

\DeclareMathOperator{\Spec}{Spec}
\DeclareMathOperator{\Ann}{Ann}
\DeclareMathOperator{\charac}{char}
\DeclareMathOperator{\mispan}{span}

\def\a{\alpha}
\def\l{\lambda}
\def\b{\beta}
\def\g{\gamma}

\begin{document}
	
	\author[Kaygorodov]{Ivan Kaygorodov}
	\address{[I.\ Kaygorodov] 
		Centro de Matemática e Aplicações, Universidade da Beira Interior, Covilh\~{a}, Portugal; \ 
		 Moscow Center for Fundamental and Applied Mathematics,      Russia; \
		Saint Petersburg  University, Russia.}	
	\email{kaygorodov.ivan@gmail.com}
	
	\author[Mart\'{i}n]{C\'{a}ndido Mart\'{i}n Gonz\'{a}lez}
	\address{[C. Mart\'{\i}n Gonz\'{a}lez] 
		Departamento de \'Algebra Geometr\'{\i}a y Topolog\'{\i}a, Fa\-cultad de Ciencias, Universidad de M\'alaga, Campus de Teatinos s/n, 29071 M\'alaga, Spain.} 
	\email{candido\_m@uma.es}

	\author[P\'aez-Guill\'an]{Pilar P\'aez-Guill\'an}
	\address{[P.\ P\'aez-Guill\'an] Fakultät für Mathematik, Universität Wien, Oskar-Morgenstern-Platz 1, 1090 Wien, Austria}
	\email{maria.pilar.paez.guillan@univie.ac.at}

	\title[Central extensions of axial algebras]{Central extensions of axial algebras}

	\begin{abstract}
		In this article, we develop a further adaptation of the method of Skjelbred-Sund to construct central extensions of  axial algebras.
		We use our method to prove that all axial central extensions (with respect to a maximal set of axes) of complex simple finite-dimensional Jordan algebras are split, and that all non-split axial central extensions of dimension $n\leq 4$ over an algebraically closed field of characteristic not $2$ are Jordan. Also, we give a classification of $2$-dimensional axial algebras and describe some important properties of these algebras.
	\end{abstract}
	
	\subjclass[2020]{17A36,17A60,17C20,17C27,17D99}
	\keywords{Axial algebra; central extension; Jordan algebra; classification.}
	
	\maketitle

	\section*{Introduction}
	Axial algebras are a recent class of non-associative commutative algebras introduced by Hall,
	Rehren and Shpectorov~\cite{hrs} in 2015. They can be seen as a certain generalisation of commutative, associative algebras, and as a common frame for Majorana algebras~\cite{hrs,why}, Jordan algebras~\cite{hrs2,hss,ms22} and other types of algebras appearing in mathematical physics. They are also related to code algebras~\cite{cmi}.

	The relevance of Majorana and axial algebras lies on the fact that they provide an axiomatic approach to some properties of vertex operator algebras (VOAs), complex algebraic structures rooted in theoretical physics. In mathematics, the best-known VOA is the moonshine $V^{\#}$, constructed by  Frenkel, Lepowsky and
	Meurman in~\cite{flm}, and whose automorphism
	group is the Monster $M$, the largest sporadic finite simple group. This object shows a link to the theory of modular functions, and was key in the proof of Borcherds~\cite{bor92} of the monstrous moonshine conjecture on the connection between the Monster and modular functions. The rigorous development of the theory of VOAs, an important tool for the proof, is also due to Borcherds~\cite{bor}.
	
	After the cited paper of Hall, Rehren and Shpectorov~\cite{hrs}, it began a systematic study of axial algebras. 
	An interesting and active direction in this study is the description of $n$-generated axial algebras of a certain type.
	So,  two-generated primitive axial algebras of Jordan type $\eta$ over fields of characteristic different from  $2$ were classified in~\cite{hrs2} by Hall, Rehren and Shpectorov.
	Rehren proved in \cite{reh1,reh2} that the dimension of  two-generated primitive axial algebras of Monster type $\left(\alpha, \beta\right)$ does not exceed $8$ if the characteristic of the ground field is not $2$ and $\alpha \notin \left\{2\beta, 4\beta\right\}$.
	Later, Franchi, Mainardis  and  Shpectorov constructed  an infinite dimensional two-generated primitive axial algebra of Monster type $\left(2, \frac{1}{2}\right),$ today  known as {\it Highwater algebra} \cite{HW},
	and they classified all  two-generated primitive axial algebras of Monster type $\left(2\beta,\beta\right)$ over a field of characteristic not $2$ in \cite{HW2}. Also, a classification of primitive symmetric two-generated axial algebras of Monster type was given between Yabe's~\cite{Y}, Franchi and Mainardis'~\cite{fm5} and Franchi, Mainardis and McInroy's~\cite{fmm}.
	On the other hand, Gorshkov and Staroletov 
	showed that a three-generated primitive axial algebra of Jordan type has dimension at most $9$ \cite{gs}; Khasraw,  McInroy and Shpectorov enumerated all the three-generated primitive axial algebras of Monster type $\left(\alpha, \beta\right)$ of a certain subclass, the so-called $4$-algebras~\cite{kmis2}.

	We cite some other directions in the research on axial algebras. Khasraw,  McInroy and Shpectorov described the structure of axial algebras \cite{axial}. 
	De Medts and Van Couwenberghe   introduced axial representations of groups and modules over axial algebras as new tools to study axial algebras \cite{dmvc}. 
	Axial algebras have been also studied from a computational approach in McInroy and Shpectorov's~\cite{mis} (see also~\cite{seress}), and from a categorical point of view in De Medts, Peacock, Shpectorov and Van Couwenberghe's~\cite{dmpsvc}. Also, a non-commutative version of axial algebras was considered by 
	Rowen and Segev \cite{rs1,rs2,rs3}.
	
	On the other hand, 
	the study of algebras generated by idempotents has a proper interest.
	Rowen and Segev described all associative and Jordan algebras generated by two idempotents~\cite{RS};
	Brešar proved that a finite-dimensional (unital) algebra is zero product determined if and only
	if it is generated by idempotents~\cite{br}, and so on.

	The paper is organised as follows. The introductory Section~\ref{s:prel} provides some basic definitions about axial algebras. We also give a classification of complex two-dimensional axial algebras and describe some of the main properties of these algebras.
	Section~\ref{s:ce} is devoted to a detailed explanation of an adaptation of  the Skjelbred-Sund method~\cite{ss78} for constructing central extensions of axial algebras:
	we describe the conditions that ensure that a given central extension of an axial algebra will also be axial (Theorem~\ref{th:gen}) and prove that 
	an axial algebra with non-zero annihilator is a central extension of an axial algebra of smaller dimension
	(Theorem~\ref{th:imp}).
	In Section~\ref{s:jordan}, we apply the methods developed in Section~\ref{s:ce} to prove
	that  a complex  finite-dimensional simple Jordan algebra  does not have  non-split axial central extensions with respect to a  maximal set of axes, and also that all the axial central extensions of Jordan algebras of dimension $n\leq 4$ over an algebraically closed field of characteristic not $2$ are again Jordan algebras.
	
	Although all the examples throughout the paper are already known in the literature, it is our hope that this technique will allow us to find new examples in the near future.
	
	\
	
	Unless otherwise stated, all algebras and vector spaces throughout the paper are assumed to be of arbitrary dimension and over an arbitrary field $F$. The generating sets for the algebras are assumed to be finite. Also, we will denote the direct sum of vector spaces by $\dot{+}$, while $\oplus$ will be reserved to denote the direct sum of algebras. Finally, we will employ $\subseteq$ for denoting non-strict inclusions, and $\subset$ for the strict ones.

	\section{Preliminaries on axial algebras}\label{s:prel}

	Let $F$ be a field, $\mathcal{F}\subseteq F$ a subset, and $\star\colon \mathcal{F}\times \mathcal{F}\to 2^{\mathcal{F}}$ a symmetric binary operation. The pair $(\mathcal{F},\star)$ is called a fusion law over $F$, and will be denoted simply by $\mathcal{F}$ whenever there is not possibility of confusion. We say that the fusion law $(\mathcal{F},\star)$ is contained in the fusion law $(\mathcal{G},\odot)$ if $\mathcal{F}\subseteq \mathcal{G}$ and, for every $\lambda,\mu\in\mathcal{F}$, it holds that $\lambda\star \mu \subseteq \lambda\odot \mu$. Also, if $\mathcal{F}\subseteq\mathcal{G}$, we will denote by $(\mathcal{G},\star)$ the fusion law resulting from setting $\lambda\star\mu=\emptyset$ for every $\lambda\in\mathcal{G}\setminus\mathcal{F}$ and every $\mu\in\mathcal{G}$.
	
	The values of any fusion law $(\mathcal{F},\star)$ can be displayed in a symmetric square table. This is the most common way to represent them; we employ it in Tables~\ref{t:m} and~\ref{table:j12}. Following the conventions in the literature (e.g. \cite{hrs},\cite{hrs2}), we will abuse notation in the tables by not writing the set symbols in unitary sets, or using a blank entry to mean the empty set. On other occasions, we will limit to write explicitly the relevant products of the fusion laws, as in Tables~\ref{t:flax2} and~\ref{t:eax2}.
	
	Let $\A$ be a commutative algebra. For any element $x\in \A$, we denote by $\Spec(x)$ the spectrum of the endomorphism $L_x \colon \A\to\A$, $y\mapsto xy$, 
	and by $\A_{\lambda}^x$ the eigenspace associated with an eigenvalue $\lambda \in \Spec(x)$. If $\mu\notin \Spec(x)$, we assume $\A_{\mu}^x=0$. Also, given a subset $\mathcal{S}\subseteq \Spec(x)$, we denote $\A_{\mathcal{S}}^x=\dot{+}_{\lambda \in \mathcal{S}}\A_{\lambda}^x$ and $\A_{\emptyset}^x=\{0\}$.
	
	Let $(\mathcal{F},\star)$ be a fusion law over $F$. An element $a\in \A$ is called an $\mathcal{F}$-axis if the following conditions hold:
	
	\begin{enumerate}
		\item $a$ is idempotent;
		\item $a$ is semisimple;
		\item $\Spec(a)\subseteq\mathcal{F}$ and $\A_{\lambda}^a\A_{\mu}^a\subseteq \A_{\lambda\star \mu}^a$, for all $\lambda, \mu \in \Spec(a)$.
	\end{enumerate}
	
	Recall that, if $\A$ has finite dimension, by $a$ being semisimple we just mean that  $ L_a$ is diagonalisable. In the infinite-dimensional case, $a$ must satisfy the next two conditions:
	
	\begin{enumerate}
		\item[(i)] For every $x\in\A$, there exists a finite-dimensional subspace $W_x\subset \A$ stable by $ L_a$ such that $x\in W_x$.
		\item[(ii)] For every subspace $W\subset \A$ of finite dimension stable by $ L_a$, the restriction $ L_a\vert_W$ is  diagonalisable.
	\end{enumerate}
	
	An $\mathcal{F}$-axial algebra over $F $ is a pair $(\A,X)$, where $\A$ is a commutative algebra over $F$ and $X$ is a finite set of $\mathcal{F}$-axes that generate $\A$.
	If the fusion law is clear, we will simply refer to axes and axial algebras.

	We will now recall some basic definitions regarding axial algebras. For more information, see for instance \cite{hrs,hrs2,hss,axial,ms23}. 
	
	Note that, from conditions $(1)$ and $(3)$ above, any axis $a \in \A$ satisfies $\A_1^a\neq 0$. An axis $a\in \A$ is called primitive if $\dim\A_1^a=1$. If an axial algebra $(\A,{\rm X})$ is generated by primitive axes, then it is called primitive. In this case, $\A_{1}^a\A_{\lambda}^a\subseteq \A_{\lambda}^a$, for all $\lambda\in \Spec(a)$.
	
	A two-generated axial algebra $(\A,\{a_1,a_2\})$ is called symmetric if it admits a flip, i.e. if there exists an automorphism switching the generating axes $a_1$ and $a_2$.
	
	An axial algebra $(\A,{\rm X})$ is said to be $m$-closed if $\A$ is spanned by products of axes of ${\rm X}$ of length at most $m$.
	
	Also, we say that an axial algebra $(\A,{\rm X})$ admits a Frobenius form if there exists a (non-zero) bilinear form $(\cdot,\cdot)\colon \A\times\A\to F$ which associates with the product of $\A$, i.e. \[(x,yz)=(xy,z)\] for all $x,y,z\in\A$. Note that this form is necessarily symmetric  \cite[Proposition 3.5]{hrs}.
	
	The radical $R(\A,{\rm X})$ of a primitive axial algebra $(\A,{\rm X})$ is the unique largest ideal of $\A$ containing no axes from ${\rm X}$. If $(\A,{\rm X})$ admits a Frobenius form, the radical of the form and $R(\A,{\rm X})$ are closely related (see~\cite{axial}).
	
	Sometimes, the fusion law $(\F,\star)$ is graded by a finite abelian group $T$, in the sense that there exists a partition $\{\F_t\mid t\in T\}$ such that for all $s,t\in T$, \[\F_s\star\F_t\subseteq \F_{st}.\] In these cases, it is induced a $T$-grading on $\A$ for each axis $a$, namely $\A=\dot{+}_{t\in T}\A^a_{\F_t}$. 
	
	Let $T^*$ be the group of linear characters of $T$. For each axis $a$, there exists a group homomorphism, $\tau_a\colon T^* \to \Aut(\A)$, where $\tau_a(\chi)\colon \A \to \A $, with $\chi\in T^*$, is defined by the linear extension of
	\begin{longtable}{rcl}
		$\tau_a(\chi)\colon \A$ & $\to$&$ \A$\\
		$u$ & $\longmapsto$&$ \chi(t)u,$ 
	\end{longtable}
	\noindent for $u\in \A_t^a$.
	
	The automorphisms of the type $\tau_a(\chi)$ are called Miyamoto automorphisms, and the image $T_a\coloneqq \Ima \tau_a$ is called the axial subgroup of $\Aut(\A)$ corresponding to $a$. If $\Za$ is a set of axes of $\A$, the subgroup \[G(\Za)\coloneqq \langle T_a \mid a \in \Za\rangle \subseteq \Aut(\A)\]
	is known as the Miyamoto group of $\A$ with respect to $\Za$.
	
	In this paper, we will restrict to dealing with gradings by the cyclic group $C_2$ of order two.  To avoid pathological situations, on many occasions we will assume that $\charac F \neq 2$. Note that, in this setting, $T_a\subseteq \Aut (\A)$ has order two for every axis $a$. Let us write $T_a=\{\id_{\A},\phi_a\}$.

	A set of axes $\Za$ is called closed if $\phi(\Za)=\Za$ for any $\phi\in G(\Za)$. The minimal closed set of axes containing $\Za$ exists and it is called the closure of $\Za$, $\overline{\Za}$.

	As an example, we select here the algebras of dimension $2$ over $\mathbb{C}$ which are axial from the classification in \cite{kv16}, and provide some information about their basic features.
	
	\begin{exam}\label{ex:ax2}
		Consider the set $\left\{\left(\a,\b,\frac{\a\b-(\a-1)(\b-1)}{4\a\b-1}\right) \mid \a,\b\in\mathbb{C}, \a,\b\neq\frac{1}{2}, \a\b\neq\frac{1}{4}\right\}$. The symmetric group $S_3$ of permutations of three elements 
		acts canonically on ${\mathbb C}^3$ and induces an action on 
		\[\left\{\left(\a,\b,\frac{\a\b-(\a-1)(\b-1)}{4\a\b-1}\right) \mid \a,\b\in\mathbb{C}, 
		\a,\b\neq\frac{1}{2}, \a\b\neq\frac{1}{4}\right\}.\] We can choose a set of representatives of 
		the orbits, $\Delta$, such that $\b\neq 0,1$, $\a+\b\neq 1$ and $\a\neq\b\neq \frac{\a\b-(\a-1)(\b-1)}{4\a\b-1}\neq \a $. Set \[\kappa=\left\{(\a,\b)\in \mathbb{C}^2\mid \a,\b\neq\frac{1}{2}, \a\b\neq\frac{1}{4},\left(\a,\b,\frac{\a\b-(\a-1)(\b-1)}{4\a\b-1}\right)\in\Delta \right\}.\] 
		Also, the cyclic group $C_2$  acts on $\mathbb{C}\setminus\{0,1\}$ by taking $\leftidx^{{-1}}(\a)=\a^{-1}$. We will fix a certain set of representatives of the orbits under this action and denote it by $\mathbb{C}^*_{>1}$.

		\begin{longtable}{lllll}
			$A$ & $E_1(0,0,0,0)$ & $e_1e_1=e_1$ & $e_1e_2= 0$ & $e_2e_2=e_2;$  \\
			$B$ & $E_1(-1,-1,-1,-1)$ & $e_1e_1=e_1$ &  $e_1e_2= - e_1 - e_2$ & $e_2e_2=e_2;$ \\
			$C(\a)_{\a\neq 0,\pm \frac{1}{2},\pm 1}$ & $E_1(\a,\a,\a,\a)_{\a\neq 0,\pm \frac{1}{2},\pm 1}$ & $e_1e_1=e_1$ &  $e_1e_2= \a (e_1 + e_2)$ & $e_2e_2=e_2;$ \\
			$D(\b)_{(0,\b)\in \kappa}$ & $E_1(0,\b,0,\b)_{(0,\b)\in \kappa}$ & $e_1e_1=e_1$ & $e_1e_2= \b e_2$  & $e_2e_2=e_2;$ \\
			$E(\a,\b)_{(\a,\b)\in \kappa,\a\neq 0}$ & $E_1(\a,\b,\a,\b)_{(\a,\b)\in\kappa,\a\neq 0}$ & $e_1e_1=e_1$ &  $e_1e_2= \a e_1 +\b e_2$ & $e_2e_2=e_2;$ \\
			$F$ & $E_2(\frac{1}{2},0,0)$ & $e_1e_1=e_1$ & $e_1e_2=\frac{1}{2} e_1 $  & $e_2e_2=e_2;$ \\
			$G(\b)_{\b\neq 0, \frac{1}{2}, 1}$ & $E_2(\frac{1}{2},\b,\b)_{\b\neq 0, \frac{1}{2}, 1}$ & $e_1e_1=e_1$ & $e_1e_2=\frac{1}{2} e_1 + \b e_2 $  & $e_2e_2=e_2;$ \\
			$H(\g)_{\g\in\mathbb{C}^*_{>1}, \g\neq 2}$ & $E_3(\frac{1}{2},\frac{1}{2},\g)_{\g\in\mathbb{C}^*_{>1}, \g\neq 2}$ & $e_1e_1=e_1$ &  $e_1e_2=\frac{1}{2} \g e_1 + \frac{1}{2\g} e_2 $ & $e_2e_2=e_2;$ \\
			$I$ & $E_5(\frac{1}{2})$ & $e_1e_1=e_1$ & $e_1e_2= \frac{1}{2}(e_1 + e_2)$  & $e_2e_2=e_2.$
		\end{longtable}
		
		All the previous algebras, alongside with certain sets of generating axes, are two-closed axial algebras. The following Table~\ref{t:ax2} summarises information about some of their basic features.
		
		We limit to provide minimal set of axes ${\rm X}$ such that $(\A, {\rm X})$ is axial. The fusion laws are displayed in Table~\ref{t:flax2}, with the following conventions: we only write the non-zero products, and we assume that $1\star \lambda=\lambda$ for all $\lambda\in\F$, $\lambda\neq 0$. Write 
		\begin{longtable}{ccr}
			$a_3$ & $=$ & $e_1+e_2;$\\
			$a_4$ & $=$ & $-(e_1+e_2);$\\
			$a_5$ & $=$ & $\frac{1}{1+2\a}(e_1+e_2);$\\
			$a_6$ & $=$ & $e_1+(1-2\b)e_2;$\\
			$a_7$ & $=$ &$\frac{1-2\a}{1-4\a\b} e_1+\frac{1-2\b}{1-4\a\b}e_2;$\\
			$a_{\a}$ & $=$ & $\a e_1 + (1-\a)e_2,$
		\end{longtable} \noindent where $\a,\b$ are elements of $\mathbb{C}$ whose requirements vary from case to case, and denote by $\mathfrak{F}_2$ the free group generated by two involutions. Note that, regarding Frobenius forms, we just offer an example for each axial algebra. 
		
		{
			\begin{longtable}{|c|c|c|c|l|}
				\hline
				$\A$  & ${\rm X}$ & $(\F,\star)$ & Symmetric   & \multicolumn{1}{c|}{$(\cdot,\cdot)$} \\
				\hline \hline \multirow{3}{*}{$A$} & $\{e_1,e_2\}$ & \multirow{3}{*}{$(\F_A,\star_A)$} & Yes & \multirow{3}{*}{$\begin{array}{l}
						(e_1,e_1)=1\\
						(e_1,e_2)=0\\
						(e_2,e_2)=1
					\end{array}$} \\ \cline{2-2} \cline{4-4}
				& $\{e_1,a_3\}$ & & \multirow{2}{*}{No}    & \\ \cline{2-2}
				& $\{e_2,a_3\}$ & &  &   \\
				
				\hline \hline  \multirow{3}{*}{$B$} & $\{e_1,e_2\}$ & \multirow{3}{*}{$(\F_B,\star_B)$} & \multirow{3}{*}{Yes} &  \multirow{3}{*}{$\begin{array}{l}
						(e_1,e_1)=-2\\
						(e_1,e_2)=1\\
						(e_2,e_2)=-2
					\end{array}$} \\ \cline{2-2}
				& $\{e_1,a_4\}$ &  &  &   \\ \cline{2-2}
				& $\{e_2,a_4\}$ &  &  & \\
				
				\hline \hline
				\multirow{3}{*}{$C(\a)$} & $\{e_1,e_2\}$ & $(\F_{C1},\star_{C1})$ & Yes   & \multirow{3}{*}{$\begin{array}{l}
						(e_1,e_1)=(1-\a)/\a\\
						(e_1,e_2)=1\\
						(e_2,e_2)=(1-\a)/\a
					\end{array}$} \\ \cline{2-4}
				& $\{e_1,a_5\}$ & \multirow{2}{*}{$(\F_{C2},\star_{C2})$} & \multirow{2}{*}{No}  &    \\ \cline{2-2}
				& $\{e_2,a_5\}$ & &  &  \\
				
				\hline \hline
				\multirow{3}{*}{$D(\b)$} & $\{e_1,e_2\}$ & $(\F_{D1},\star_{D1})$ & \multirow{3}{*}{No} &  \multirow{3}{*}{$\begin{array}{l}
						(e_1,e_1)=1\\
						(e_1,e_2)=0\\
						(e_2,e_2)=0
					\end{array}$} \\ \cline{2-3}
				& $\{e_1,a_6\}$ & $(\F_{D2},\star_{D2})$ &   &    \\  
				\cline{2-3}  & $\{e_2,a_6\}$ & $(\F_{D3},\star_{D3})$ & & \\
				
				\hline \hline
				\multirow{3}{*}{$E(\a,\b)$} & $\{e_1,e_2\}_{(\a\neq 1)}$ & $(\F_{E1},\star_{E1})$ & \multirow{3}{*}{No}& \multirow{3}{*}{$\begin{array}{l}
						(e_1,e_1)=(1-\b)/\a\\
						(e_1,e_2)=1\\
						(e_2,e_2)=(1-\a)/\b
					\end{array}$} \\ \cline{2-3} 
				& $\{e_1,a_7\}$ & $(\F_{E2},\star_{E2})$ &   &    \\ \cline{2-3}
				& $\{e_2,a_7\}_{(\a\neq 1)}$  & $(\F_{E3},\star_{E3})$ &  &   \\
				
				\hline \hline $F$ & $\{e_1,e_2\}$ & $(\F_{F},\star_{F})$ & No  & $\begin{array}{l}
					(e_1,e_1)=0\\
					(e_1,e_2)=0\\
					(e_2,e_2)=1
				\end{array}$\\
				
				\hline \hline $G(\b)$ & $\{e_1,e_2\}$ & $(\F_{G},\star_{G})$ & No  & $\begin{array}{l}
					(e_1,e_1)=2(1-\b)\\
					(e_1,e_2)=1\\
					(e_2,e_2)=1/2\b
				\end{array}$ \\
				
				\hline \hline $H(\g)$ & $\{e_1,e_2\}$ & $(\F_{H},\star_{H})$ & $ \g=-1$ &  $\begin{array}{l}
					(e_1,e_1)=(2\g -1)/\g^2\\
					(e_1,e_2)=1\\
					(e_2,e_2)=(2-\g)\g
				\end{array}$ \\
				
				\hline \hline $I$ & $\{a_{\a},a_{\b}\}_{\a\neq\b}$ & $(\F_{I},\star_{I})$ & $ \a\neq-\b$  & $\begin{array}{l}
					(e_1,e_1)=1\\
					(e_1,e_2)=1\\
					(e_2,e_2)=1
				\end{array}$ \\
				\hline 
				
				\caption{Complex axial algebras of dimension $2$.}
				\label{t:ax2}
		\end{longtable}}
		\begin{longtable}{|c|c|ccc|}
			\hline
			Fusion law & $\F$ &	 \multicolumn{3}{c|}{}\\
			\hline
			$(\F_{A},\star_{A})$ & $ \{1,0\} $& $ 0 \star_{A} 0 =   0 $ && \\
			
			\hline
			$(\F_{B},\star_{B})$ & $ \{1,-1\} $&  $(-1) \star_{B} (-1) =  1   $ && \\
			
			\hline	
			$(\F_{C1},\star_{C1})$ & $ \{1,\a \} $&  $\a \star_{C1} \a = \{ 1,\a\} $ && \\

			\hline
			$(\F_{C2},\star_{C2})$ & $ \{1,\a, \l \} $& 
			
			$\a \star_{C2} \a = \{ 1,\a\},$ & 
			$\l \star_{C2} \l =  1$ & \\
			
			\hline
			$(\F_{D1},\star_{D1})$ & $ \{1, \beta, 0\} $& $\beta \star_{D1} \beta =   \beta  , $& 
			$0 \star_{D1} 0 = \{1,0\} $ &\\

			\hline
			$(\F_{D2},\star_{D2})$ & $ \{1, \beta, 1-\beta \} $&  $\beta \star_{D2} \beta =   \beta $, & 
			$(1-\beta) \star_{D2} (1-\beta) = 1-\beta$ &  \\

			\hline
			$(\F_{D3},\star_{D3})$ & $ \{1, 1-\beta, 0 \} $&  
			$(1-\beta) \star_{D3} (1-\beta) =   1-\beta $, & $0 \star_{D3} 0 = \{1,0\}$  & \\

			\hline
			$(\F_{E1},\star_{E1})$ & $ \{1, \beta, \l \} $&  
			$\beta \star_{E1} \beta =   \{ 1, \beta \}, $& 
			$\l \star_{E1} \l = \{1,\l\} $ &\\

			\hline
			$(\F_{E2},\star_{E2})$ & $ \{1, \a, \beta  \} $&  
			$\a \star_{E2} \a =   \{ 1, \a \}, $& 
			$\beta \star_{E2} \beta = \{1,\beta \} $ &\\

			\hline
			$(\F_{E3},\star_{E3})$ & $ \{1, \a, \l  \} $&  
			$\a \star_{E3} \a =   \{ 1, \a \}, $& 
			$\l \star_{E3} \l = \{1,\l \} $ &\\

			\hline
			$(\F_{F},\star_{F})$ & $ \{1, 1/2, 0  \} $&  
			$1/2 \star_{F} 1/2 = 1/2 $,& $0 \star_{F} 0 =   \{ 1, 0 \}$
			& \\

			\hline
			$(\F_{G},\star_{G})$ & $ \{1, \beta, 1/2  \} $&  
			$\beta \star_{G} \beta =   \{ 1, \beta \}, $& 
			$1/2 \star_{G} 1/2 = \{ 1, 1/2\} $ &\\

			\hline
			$(\F_{H},\star_{H})$ & $ \{1,  \frac{1}{2\gamma}, \gamma/2  \} $&  
			$ \frac{1}{2\gamma} \star_{H}  \frac{1}{2\gamma} =   \{ 1,  \frac{1}{2\gamma} \}, $& 
			$\gamma/2  \star_{H} \gamma/2  = \{ 1,\gamma/2 \} $ &\\

			\hline
			$(\F_{I},\star_{I})$ & $ \{1, 1/2   \} $&  
			&  &\\

			\hline
			\multicolumn{5}{c}{} \\
			
			\caption{Fusion laws in Table~\ref{t:ax2}.}
			\label{t:flax2}
		\end{longtable}
		(For $D3$, $\l=\frac{1}{1+2\a}$, and for $E1$ and $E3$,
		$\lambda=\frac{1-\alpha-\beta}{1-4\alpha\beta}$).

		We indicate now some other properties of the above algebras not displayed in Table~\ref{t:ax2}. The only algebras which are not primitive are $(A,\{e_1,a_3\})$ and $(A,\{e_2,a_3\})$; the only ones that have non-zero radical, $(D(\b),\{e_1,a_6\})$ and $(I,\{a_{\a},a_{\b}\})$, with $R(D(\b),\{e_1,a_6\})=\langle e_2\rangle$ and $R(I,\{a_{\a},a_{\b}\})=\langle e_1-e_2\rangle$, respectively. For the sake of completeness, we point out that the algebras $A$, $B$, $C(\a)$, $E(\a,\b)$, $G(\b)$ and $H(\g)$ are in fact simple. 
		
		Note that, between the former fusion laws, only $(\F_B,\star_B)$, $(\F_{C2},\star_{C2})$ and $(\F_I,\star_I)$  admit $\zd$-gradings. We now show in detail the explicit Miyamoto groups of $B$, $C(\a)$ and $I$.
		
		Consider first $B$. The fusion law  $(\F_B,\star_B)$ admits the $\zd$-grading $(\F_B)_1=\{1\}$, $(\F_B)_{-1}=\{-1\}$. The Miyamoto automorphisms with respect to the axes $e_1$, $e_2$ and $a_4$ are
		\begin{longtable}{rcl}
			$\phi_{e_i}\colon  B$ & $\to$&$ B$\\
			$e_i$ & $\mapsto$&$ e_i,$\\
			$e_j$ & $\mapsto$&$ -e_i-e_j,$
		\end{longtable}
		\noindent for $i,j\in\{1,2\}$, $i\neq j$, and
		\begin{longtable}{rcl}
			$\phi_{a_4}\colon  B$  & $\to$&$ B$\\
			$e_1$ & $\mapsto$&$ e_2,$\\
			$e_2$ & $\mapsto$&$ e_1.$
		\end{longtable}
		
		\noindent As a consequence, the Miyamoto group with respect to ${\rm X}=\{e_1,e_2\}$ is
		\[G({\rm X})=\langle \phi_{e_1},\phi_{e_2} \mid \phi_{e_1}^2=\phi_{e_2}^2=(\phi_{e_1}\phi_{e_2})^3=(\phi_{e_2}\phi_{e_1})^3=1\rangle\simeq S_3, \] and also, denoting ${\rm X}_{i}=\{e_i,a_4\}$,
		\[G({\rm X}_{i})=\langle \phi_{e_i},\phi_{a_4} \mid \phi_{e_i}^2=\phi_{a_4}^2=(\phi_{e_i}\phi_{a_4})^3=(\phi_{a_4}\phi_{e_i})^3=1\rangle\simeq S_3, \]
		for $i\in\{1,2\}$, $i\neq j$. The closure of any of the generating sets of axes considered in Table~\ref{t:ax2} is $\overline{X}=\{e_1,e_2,a_4\}$.
		\newline
		
		Regarding $C(\a)$, the fusion law $(\F_2,\star_2)$ admits the $\zd$-grading $(\F_{C2})_{1}=\{1,\a\}$, $(\F_{C2})_{-1}=\{\l\}$. The Miyamoto automorphisms with respect to the axes $e_1$ and $e_2$ are $\id_{C(\a)}$, whereas
		\begin{longtable}{rcl}
			$\phi_{a_5}\colon  C(\a)$ & $\to$&$ C(\a)$\\
			$e_1$ & $\mapsto$&$ e_2$,\\
			$e_2$ & $\mapsto$&$ e_1$.
		\end{longtable}
		
		\noindent Then, the Miyamoto groups with respect to ${\rm X}_{i}=\{e_i,a_5\}$ are
		\[G({\rm X}_{i})=\langle \phi_{a_5} \mid \phi_{a_5}^2=1\rangle\simeq \zd, \]
		for $i\in\{1,2\}$. The closure of both generating sets of axes considered in Table~\ref{t:ax2} is $\overline{X}=\{e_1,e_2,a_5\}$.
		
		Write ${\rm X}=\{e_1,e_2\}$. Since $(\F_{C1},\star_{C1})\subseteq (\F_{C2},\star_{C2})$, we could also consider $(C(\a),{\rm X})$ as an $(\F_{C2},\star_{C2})$-axial algebra. In this case, the Miyamoto group reduces to $G({\rm X})=\{\id_{C(\a)}\}$, and $\overline{{\rm X}}={\rm X}$.
		
		Finally, consider $I$. Now, the $\zd$-grading of $(\F_I,\star_I)$ is  $(\F_{I})_{1}=\{1\}$, $(\F_{I})_{-1}=\{\frac{1}{2}\}$. The Miyamoto group with respect to ${\rm X}=\{a_{\a},a_{\b}\}_{\a\neq\b}$ is \[G({\rm X})=\langle \phi_{a_{\a}},\phi_{a_{\b}} \mid \phi_{a_{\a}}^2=\phi_{a_{\b}}^2=1\rangle\simeq \mathfrak{F}_2, \]
		where
		\begin{longtable}{rcl}
			$\phi_{a_{\a}}\colon  I$ & $\to$ &$I$\\
			$e_1$ & $\mapsto$&$ (2\a-1)e_1 + 2(1-\a)e_2,$\\
			$e_2$ & $\mapsto$&$ 2\alpha e_1 + (1-2\alpha)e_2,$
		\end{longtable}
		\noindent for any $\a\in\mathbb{C}$. Also, it holds that $\overline{X}=\{a_{\a+\mathbb{Z}(\a-\b)}\}$.
	\end{exam}

	\section{Central extensions of axial algebras}\label{s:ce}
	Let $\A$ be an algebra and ${\rm V}$ a vector space, and recall that a central extension of $\A$ by $\rm V$ is a short exact sequence $0\to {\rm V}\to \Ba \to \A\to 0$ such that the image of $\rm V$ is contained in the annihilator $\Ann(\Ba)$; we say that the dimension of the extension is $\dim {\rm V}$. It is customary to identify the central extension with $\Ba$. Let also $\theta\colon \A \times \A \to {\rm V}$ be a bilinear map, and define $\A_{\theta}=\A\dot{+} {\rm V}$, which can be given structure of algebra with the product $[x+v,y+w]_{\theta}=xy+\theta(x,y)$. It is immediate to check that $\A_{\theta}$ is a central extension of $\A$ with respect to ${\rm V}$. Also, it can be seen that  every central extension of $\A$ of arbitrary dimension is isomorphic to $\A_{\theta}$ for some ${\rm V}$ and $\theta$ in the above conditions. Then, we will identify any central extension of $\A$ with a certain $\A_{\theta}$ for some ${\rm V}$ and $\theta$ in the above conditions.
	
	In the following paragraphs, we present some basic features of this approach to central extensions which will be needed later (for more details about the Skjelbred-Sund method, see, for example, \cite{cald,KLP} and references therein).
	
	Let $f\colon \A \to {\rm V}$ be a linear map and define the bilinear map $\delta f\colon \A\times \A \to {\rm V}$ by $\delta f (x,y)=f(xy)$. The set $\{\delta f\mid f\colon \A\to {\rm V} \text{ is linear}\}$ is a linear subspace of the bilinear maps from $\A$ to ${\rm V}$, so we can consider the quotient space. Note that if $\theta'=\theta + \delta f$ for some $f\colon \A \to {\rm V}$, then the map $\varphi\colon \A_{\theta}\to \A_{\theta'}$ defined by $\varphi(x+v)=x+v+f(x)$ is an isomorphism. Therefore, the isomorphy class of $\A_{\theta}$ does not depend on the representatives $\theta$ of the equivalence class $[\theta]$. Also, there is a natural (right) action of $\Aut(\A)$ on this quotient space, induced by $\phi \theta(x,y)=\theta(\phi(x),\phi(y))$ for $\phi \in \Aut(\A)$. The isomorphy class of $\A_{\theta}$ does not depend either on the particular point in the orbit of $[\theta]$ under this action, as the map $\varphi\colon \A_{\theta}\to \A_{\phi\theta}$ given by $\varphi(x+v)=\phi^{-1}(x)+v$ is an isomorphism. 
	
	
	Given a bilinear map $\theta\colon \A \times \A \to {\rm V}$ and a basis $\{e_{\g}\}_{\g\in\Gamma}$ of ${\rm V}$, there exist $\lvert \Gamma \rvert$ unique  bilinear maps $\theta_{\g}\colon \A \times \A \to F$ such that $\theta(x,y)=\sum_{\g\in\Gamma}\theta_{\g}(x,y)e_{\gamma}$.
	A central extension $\A_{\theta}$ is said to have an annihilator component if there exist an ideal $I$ and a non-zero subalgebra of $\Ann \A_{\theta}$, $J$, such that $\A_{\theta}=I\oplus J$. The central extensions with annihilator component are called  split; without annihilator component,   non-split. If the dimensions of $\A$ and ${\rm V}$ are finite, for a non-split central extension $A_{\theta}$, it holds that the set $\{[\theta_{\g}]\}_{\g\in\Gamma}$ is linearly independent; the converse is also true under the hypothesis $\Ann \A_{\theta}={\rm V}$. 
	
	In the subsequent, we study when a central extension of an axial algebra is axial in terms of the bilinear map $\theta$. Let us fix the following notation. We will denote by $L_x$ and $L_{x+v}^{\theta}$ the operators of left multiplication in $\A$ and in $\A_{\theta}$, respectively. The spectrum of $L_{x+v}^{\theta}$ will be denoted by $\Spec_{\theta}(x+v)$. Also, we write $\theta_{x}^{\perp}=\{y\in \A \mid \theta(x,y)=0\}$, and denote by $P\colon \A_{\theta}\to \A$ the natural projection onto $\A$.

	The following results are easy consequences of the definitions:
	
	\begin{lm}
		Let $\A$ be a commutative algebra, ${\rm V}$ a vector space and $\theta\colon \A\times \A \to {\rm V}$ a bilinear map. Then, $\A_{\theta}$ is commutative if and only if $\theta$ is symmetric.
	\end{lm}

	\begin{lm}\label{lm:crjone}
		Let ${\rm V}$ be a vector space and $\A$ an algebra
		with an element $x\in \A$ such that $L_x\colon \A\to \A$ is diagonalisable. Choose a symmetric bilinear map $\theta\colon \A\times \A\to {\rm V}$ such that $\ker(L_x)\subseteq \theta_x^\perp$. Then if $\{e_\a\}_{\a\in\mathcal{A}}$ is a basis of $\A$ diagonalising $L_x$ and
		$L_x(e_\a)=\l_\a e_\a$, $\l_\a\in F$, for any $\a$, we may construct a basis of $\A_\theta$
		given by $$B:=\{e_\a+\l_\a^{-1}\theta(x,e_\a)\colon\l_\a\ne 0\}\sqcup\{e_\a\colon\l_\a=0\}\sqcup B_{\rm V}$$ where $B_{\rm V}$ is any basis of ${\rm V}$. The basis $B$ diagonalises $L_{x+v}$ for any $v\in {\rm V}$.
	\end{lm}
	
	\begin{lm}\label{l:eig}
		Let $\A$ be a commutative algebra, ${\rm V}$ a vector space and $\theta\colon \A\times \A \to {\rm V}$ a bilinear map. Then, $\Spec_{\theta}(x+v)=\Spec(x)\cup\{0\}$ for all semisimple $x\in \A$ and all $v\in {\rm V}$, and the eigenspaces of $L_{x+v}^{\theta}$ are
		\[(\A_{\theta})_{\lambda}^{x+v} =\{y+\lambda^{-1}\theta(x,y)\mid y\in\A_{\lambda}^{x}\},\] for $\lambda \in \Spec(x)$, $\lambda\neq 0$, and
		\[(\A_{\theta})_0^{x+v}=\{y+w\in \A_{\theta}\mid y\in\A_0^{x}\cap\theta_{x}^{\perp}\},\] recalling that we mean $\A_0^x=0$ if $0\notin\Spec(x)$.
	\end{lm}

	Furthermore we have:
	\begin{lm}\label{lm:crjtwo}
		Let $B=\{e_\a+v_\a\}_{\a\in\mathcal A}$ be a basis of $\A_\theta$ diagonalising $L_{x+v}$. Then 
		$\ker(L_x)\subseteq\theta_x^\perp$ and $L_x$ is diagonalisable. 
	\end{lm}
	\begin{proof}
		We have $L_{x+v}(e_\a+v_\a)=\l_\a(e_\a+v_\a)$, which implies $xe_\a=\l_\a e_\a$ and $\theta(x,e_\a)=\l_\a v_\a$ for any $\a\in\mathcal{A}$. Note that when $\l_\a\ne 0$ we have $v_\a=\l_\a^{-1}\theta(x,e_\a)$. Define $S\coloneqq \{\a\in\mathcal{A}\colon \l_\a\ne 0\}$ and $T\coloneqq\mathcal{A}\setminus S$. So $B=\{e_\a+\l_\a^{-1}\theta(x,e_\a)\colon \a\in S\}\sqcup\{e_\a+v_\a\colon \a\in T\}$. The set $\{e_\a\}_{\a\in\mathcal{A}}$ is a system of linear generators of $\A$ and so a suitable subset $\{e_\a\}_{\a\in\mathcal{A}'}$ is a basis of $\A$. In this basis we can distinguish those $e_\a$'s whose $\l_\a$ is non-zero and those whose $\l_\a=0$. So we have a basis of $\A$ of the form $B'=\{e_\a\colon\l_\a\ne 0\}_{\a\in\mathcal{A}'}\sqcup\{e_\a\colon\l_\a=0\}_{\a\in\mathcal{A}'}$. Take $z\in\ker(L_x)$ and write $z=\sum_\a k_\a e_\a$ with $k_\a\in F$ relative to the basis $B'$. Then \[0=xz=\sum_\a k_\a xe_\a=\sum_\a k_\a\l_\a e_\a\] where the $\lambda_\a$'s in the last sum are those which are non-zero. Consequently $k_\a \l_\a=0$, that is, $k_\a=0$. Thus $z=\sum_\a k_\a e_\a$ where the sum is extended to those indices $\a$ for which $\l_\a=0$. So $\theta(x,z)=\sum_\a k_\a\theta(x,e_\a)=\sum_\a k_\a\l_\a v_\a=0$
		and $z\in\theta_x^\perp$. The fact that $L_x$ is diagonalisable follows now from the fact that $xe_\a=\l_\a e_\a$ for any $\a\in \mathcal{A}'$.
	\end{proof}

	Fix a fusion law $(\mathcal{F},\star)$. Unless otherwise stated, all axial algebras will be assumed to be axial with respect to $(\mathcal{F},\star)$.
	
	Let $(\A,{\rm X})$ an axial algebra, ${\rm V}$ a vector space  and $\theta\colon \A\times \A\to {\rm V}$ a symmetric bilinear map such that $\{[\theta_{\g}]\}_{\g\in\Gamma}$ are linearly independent. Let $\{{\rm X}^i\}_{i\in I}$ be the family of minimal sets of axes that generate $\mathbf A$, being ${\rm X}^i=\{a_1^i,\ldots,a_{r^i}^i\}$. In particular, each ${\rm X}^i$ is linearly independent and can be extended to a basis $B^i=\{a^i_{j}\}_{j\in J}$ of $\A$.

	
	Set $\omega^i_j=\theta(a^i_j,a^i_j)$ and define $f^i_j\colon \A\to {\rm V}$ by $f^i_j(a^i_k)=\omega^i_j\delta_{jk}$, for $j=1,\dots,r^i$ and $k\in J$. Then, consider \[\theta^i=\theta-\sum_{j=1}^{r^i}\delta f_j^i.\]
	It is clear that $[\theta]=[\theta^i]$; moreover, it holds that $\theta^i(a^i_k,a^i_k)=0$ for all $k=1,\dots,r^i$:
	
	\begin{equation} \label{eq:thid0}
		\theta^i(a^i_k,a^i_k)=\theta(a^i_k,a^i_k)-\sum_{j=1}^{r^i}\delta f_j(a^i_k,a^i_k)=\omega^i_k - \sum_{j=1}^{r^i} f_j(a^i_k)=\omega^i_k - \omega^i_k=0.
	\end{equation}
	
	For the sake of simplicity, we will drop the superscript $i$ whenever ${\rm X}$ is assumed to be a minimal set of axes generating $\A$. Also, when ${\rm X}$ is linearly independent, we can assume without loss of generality that $\theta(a_j,a_j)=0$ for all $j=1,\dots, r$.
	
	Let us establish another piece of notation. Let $a\in {\rm X}$ and $\lambda,\mu\in\Spec(a)$. For $x\in\A_{\lambda}^{a}$ and $y\in\A_{\mu}^{a}$,  write \[xy=\sum_{0\neq \nu\in\lambda\star \mu}z_{\nu}+z_0,\] where $z_{\nu}\in\A_{\nu}^{a}$ and $z_0\in\A_{0}^{a}$.

	\begin{prop}\label{prop:ss}
		Let $(\A,{\rm X})$ be an axial algebra, ${\rm V}$ a vector space and $\theta\colon \A\times \A\to {\rm V}$ a symmetric bilinear map. Let $x\in\A$ be semisimple and take $v\in {\rm V}$. Then, $x+v\in\A_{\theta}$ is semisimple if and only if \begin{equation}\label{axcond1}
			\ker L_x\subseteq \theta_{x}^{\perp}.
		\end{equation} 
		Furthermore, if an axis $a\in {\rm X}$ satisfies condition~\eqref{axcond1}, the eigenspace decomposition of $\A_{\theta}$ according to $a+v$ follows the fusion law $(\mathcal{F}\cup\{0\},\star)$
		if and only if for every  $\lambda,\mu\in\Spec(a)$ such that $0\notin\lambda\star \mu$, it holds that 
		\begin{equation}\label{axcond2}
			\theta(x,y)=\sum_{\nu\in\lambda\star \mu}\nu^{-1}\theta(a,z_{\nu})
		\end{equation}
		for all  $x\in\A_{\lambda}^{a}$, $y\in\A_{\mu}^{a}$. 
	\end{prop}

	\begin{proof}
		The first part of the proposition follows trivially from Lemmas~\ref{lm:crjone} and~\ref{lm:crjtwo}.
		
		Take $a\in {\rm X}$ satisfying condition~\eqref{axcond1} and $v\in {\rm V}$. By Lemma~\ref{l:eig}, the eigenspace decomposition of $\A_{\theta}$ according to  $a+v$ follows the fusion law $(\mathcal{F}\cup \{0\},\odot)$ for some symmetric binary operation $\odot$.
		
		Throughout the rest of the proof, we will denote $\Ba=\A_{\theta}$ for the sake of simplicity. Lemma~\ref{l:eig} gives the description of $\Ba_{\lambda}^{a+v}$ for all $\lambda\in\Spec_{\theta}(a+v)$; moreover, under the present hypotheses, we can particularise  
		\[\Ba_0^{a+v}=\{x+v\in \Ba\mid x\in\A_0^{a}\}.\]

		Take $\lambda,\mu\in\Spec(a)$, $x+u\in\Ba_{\lambda}^{a+v}$ and $y+w\in\Ba_{\mu}^{a+v}$. Then,
		\begin{align*}
			&[x+u,y+w]_{\theta}=xy+\theta(x,y)=\\
			& \sum_{0\neq \nu\in\lambda\star \mu} (z_{\nu} + \nu^{-1}\theta(a,z_{\nu}) ) + \left(z_0 + \left(\theta(x,y)-\sum_{0\neq \nu\in\lambda\star \mu}\nu^{-1}\theta(a,z_{\nu})\right)\right)\in \sum_{0\neq \nu\in\lambda\star \mu}\Ba_{\nu}^{a+v}\oplus \Ba_0^{a+v}.
		\end{align*}
		It is clear that $\Ba_{\lambda}^{a+v}\Ba_{\mu}^{a+v}\subseteq \Ba_{\lambda\star\mu}^{a+v}$ if and only if $0\in\lambda\star \mu$ or condition~\eqref{axcond2} holds for every  $x\in\A_{\lambda}^{a}$ and every $y\in\A_{\mu}^{a}.$
		
		Also, if $0\notin \Spec(a)$, $\Ba_0^{a+v}\Ba_{\lambda}^{a+v}=\Ba_0^{a+v}\Ba_0^{a+v}=\{0\}$ for all $\lambda\in\Spec(a)$. The result follows.
	\end{proof}

	Note that conditions~\eqref{axcond1} and~\eqref{axcond2} do not depend on the representative of $[\theta]$. Set $\theta'=\theta + \delta f$ for some linear map $f\colon \A\to {\rm V}$, and take $x\in \A$ and $a\in {\rm X}$. Given $y\in\ker L_x$, we have that $\delta f(x,y)=f(xy)=0$, and therefore $\theta'$ satisfies condition~\eqref{axcond1} if and only if $\theta$ does. Also, given $\lambda,\mu\in\Spec(a)$ such that $0\notin\lambda\star \mu$, we can write
	\[\delta f(x,y)=f(xy)=\sum_{\nu\in\lambda\star\mu}f(z_{\nu})=\sum_{\nu\in\lambda\star\mu}\nu^{-1}f(az_{\nu}) =\sum_{\nu\in\lambda\star\mu}\nu^{-1}\delta f(a,z_{\nu}).\]
	We conclude that $\theta'$ satisfies condition~\eqref{axcond2} if and only if $\theta$ does.

	\begin{cor}\label{cor:d2}
		Let $(\A,{\rm X})$ be a two-dimensional axial algebra, ${\rm V}$ a vector space and $\theta\colon \A\times \A\to {\rm V}$ a non-zero symmetric bilinear map.
		Take an axis $a\in {\rm X}$. Then, condition~\eqref{axcond1} is satisfied if and only if $0\notin\Spec(a)$.
	\end{cor}
	
	\begin{proof}
		Assume that $0\in\Spec(a)$ and that $\{a,b\}$ is a minimal set of axes, with $\theta(a,a)=\theta(b,b)=0$, and note that $\{a,b\}$ is also a basis of $\A$. Note that $a\in \theta_a^{\perp}$ but $b \notin \theta_a^{\perp}$, as otherwise $\theta$ would be the zero map. Then, $\theta_a^{\perp}=\langle a \rangle$. By hypothesis, $\ker L_a$ is non-zero and $a\notin \ker L_a$. It follows that condition~\eqref{axcond1} is not satisfied. The converse is trivial.
	\end{proof}
	
	We introduce now a notion of cocycles for axial algebras. Note that we do not intend to relate them to any theory of cohomology for axial algebras; instead, the choice of the term ``cocycle'' is motivated because they will help to describe the extensions of axial algebras.
	
	\begin{defn}\label{d:coc}
		Let $(\A,{\rm X})$ be an $(\mathcal{F},\star)$-axial algebra, ${\rm V}$ a vector space and $\theta\colon \A\times \A\to {\rm V}$ a symmetric bilinear map. We say that $\theta$ is a cocycle relative to a subset ${\rm X}'\subseteq {\rm X}$ if condition~\eqref{axcond1} is satisfied for all $a\in {\rm X}'$ and if, for every $\lambda,\mu\in\mathcal{F}$ such that $0\notin \lambda\star\mu$, condition~\eqref{axcond2} holds for all $a\in {\rm X}'$ such that $\lambda,\mu\in\Spec(a)$, all  $x\in\A_{\lambda}^{a}$ and all $y\in\A_{\mu}^{a}$. The vector space formed by them will be denoted by $\Za(\A,{\rm V};{\rm X}')$.
	\end{defn}
	
	The next technical lemma will be needed for the main results of this section.
	
	\begin{lm}\label{l:gen}
		Let $(\A,{\rm X})$ be an axial algebra, ${\rm V}$ a vector space and $\theta\colon \A\times \A\to {\rm V}$ a symmetric bilinear map such that $\{[\theta_{\g}]\}_{\g\in\Gamma}$ are linearly independent. Assume that ${\rm X}=\{a_j\}_{j=1}^r$ is a minimal set of axes generating $\A$. Then, ${\rm X}$ is a minimal generating set for $\A_{\theta}$.
	\end{lm}
	
	\begin{proof}
		Let us denote by $\langle {\rm X}\rangle$ the subalgebra of $\A$ generated by ${\rm X}$, by $\langle {\rm X}\rangle_{\theta}$ the subalgebra of $\A_{\theta}$ generated by ${\rm X}$.
		Once we prove that $\A\subseteq \langle {\rm X}\rangle_{\theta}$, 
		we will know that $\langle {\rm X}\rangle_{\theta}=\A_{\theta}$ 
		by the linear independence of $\{[\theta_{\g}]\}_{\g\in\Gamma}$. 
		
		Let $B=\{a_j\}_{j\in J}$ be a basis of $\A$ extending ${\rm X}$, and denote $J_r=J\setminus\{1,\dots,r\}$. Since $\A=\langle {\rm X}\rangle$,  we can express each $a_j$, for $j\in J_r$, as a finite sum $a_j=\sum_{l=1}^{m_j}\alpha_{j,l}\prod {\rm X}_{j,l}$, where $\alpha_{j,l}\in F$ and $\prod {\rm X}_{j,l}$ denotes a product of elements of ${\rm X}$ with a certain arrangement of brackets. 
		Set $\prod {\rm X}_{j,l} = (\prod {\rm X}^1_{j,l})(\prod {\rm X}^2_{j,l})$, where $\prod {\rm X}^1_{j,l}$ and $\prod {\rm X}^2_{j,l}$ are products of elements of $\rm X$ with strictly smaller length than $\prod {\rm X}_{j,l}$. Set also $\omega_j=\sum_{l=1}^{m_j}\alpha_{j,l}\theta(\prod {\rm X}^1_{j,l}, \prod {\rm X}^2_{j,l})$, and define the homomorphism $f\colon \A\to {\rm V}$ by $f(a_j)=0$ for $j=1,\dots,r$, and $f(a_j)=\omega_j$, for $j\in J_r$. Consider $\theta'=\theta-\delta f$. Since $\A_{\theta}$ and $\A_{\theta'}$ are isomorphic and ${\rm X}$ is preserved by the isomorphism, it is enough to show that  $\A\subseteq \langle {\rm X}\rangle_{\theta'}$. We proceed by induction in the largest length $L$ of the products $\prod {\rm X}_{j,l}$ for $l=1,\dots, m_j$. If $L=1$, it is trivial that $a_k\in \langle {\rm X}\rangle_{\theta'}$. In the general case,
		\begin{longtable}{rcl}
			$\sum\limits_{l=1}^{m_k}$&$\alpha_{k,l}$&$\left(\left[\prod {\rm X}^1_{k,l},\prod {\rm X}^2_{k,l}\right]_{\theta'}\right)$\\
			&$=$&$\sum\limits_{l=1}^{m_k}\alpha_{k,l}\prod {\rm X}_{k,l}+\sum\limits_{l=1}^{m_k}\alpha_{k,l}\theta'\left(\prod {\rm X}^1_{k,l}, \prod {\rm X}^2_{k,l}\right)-\sum\limits_{l=1}^{m_k}\alpha_{k,l}\left(\delta f\left(\prod {\rm X}^1_{k,l},\prod {\rm X}^2_{k,l}\right)\right)$\\
			&$=$&$a_k+\omega_k- f\left(\sum\limits_{l=1}^{m_k}\alpha_{k,l} \prod {\rm X}_{k,l}\right)
			= a_k+\omega_k-f(a_k)=a_k,$
		\end{longtable}
		\noindent and by induction $a_k\in \langle {\rm X}\rangle_{\theta'}$.
		
		Finally, we prove the minimality of ${\rm X}$. If there existed a subset ${\rm X}'\subset {\rm X}$ generating $\A_{\theta}$, $P({\rm X}')={\rm X}'\subset {\rm X}$ would be a set of axes generating $P(\A_{\theta})=\A$, a contradiction.
	\end{proof}
	
	We put together all the previous results in the following proposition.
	
	\begin{prop}\label{prop:min}
		Let $(\A,{\rm X})$ be an axial algebra, ${\rm V}$ a vector space and $\theta\colon \A\times \A\to {\rm V}$ a symmetric bilinear map such that $\{[\theta_{\g}]\}_{\g\in\Gamma}$ are linearly independent. Assume that ${\rm X}=\{a_j\}_{j=1}^r$ is a minimal set of axes generating $\A$. The pair $(\A_{\theta},{\rm X})$ is $(\mathcal{F}\cup\{0\},\odot)$-axial if and only if condition~\eqref{axcond1} holds for all $j=1,\dots,r$, for some fusion law $(\mathcal{F}\cup\{0\},\odot)$ containing $(\mathcal{F},\star)$. Furthermore, we can take $\odot=\star$ if and only if $\theta\in\Za(\A,{\rm V};{\rm X})$.
	\end{prop}

	\begin{proof}
		As ${\rm X}$ is minimal, we may assume that $\theta(a_j,a_j)=0$ for all $j=1,\dots,r$. Simply observe that
		the elements $a_j$ are idempotents in $\A_{\theta}$ for $j=1,\dots,r$, and use Lemma~\ref{l:eig}, Lemma~\ref{l:gen} and Proposition~\ref{prop:ss}.
	\end{proof}

	Note that, if $(\mathcal{F},\star)$ is a minimal fusion law for $(\A,{\rm X})$ in the sense that $(\A,{\rm X})$ is not axial for any fusion law strictly contained in
	$(\mathcal{F},\star)$, one has that $(\mathcal{F}\cup\{0\},\odot)$ is minimal for $(\A_{\theta},{\rm X})$, too.

	\begin{exam}\label{ex:eax2}
		Consider the two-dimensional axial algebras over $\mathbb{C}$ described in Example~\ref{ex:ax2}. By Corollary~\ref{cor:d2}, only $B$, $C(\a)$, $D(\b)$ $($with respect to $\{e_1,a_6\}$$)$, $E(\a,\b)$, $G(\b)$, $H(\g)$ and $I$ admit an axial central extension. We select their commutative non-split central extensions of dimension $1$ from the classification in~\cite{cfk}: all of them are given by the representative $\theta$ determined by $\theta(e_i,e_i)=0$ for $i\in\{1,2\}$ and $\theta(e_1,e_2)=1$. Now we can apply Proposition~\ref{prop:min} to find out their axial structures, shown in Table~\ref{t:eax2}. Note that $\theta\notin\Za(\A,{\rm V};{\rm X})$ in any instance. To describe the corresponding fusion laws, we follow the same conventions as in Table~\ref{t:flax2}.

		%
		
		\

		{\tiny
			\begin{longtable}{|c|c|c|cc|}
				\hline
				$\A_{\theta}$  & ${\rm X}$ & $(\G,\odot)$ & $\odot$ & \\
				
				\hline \hline  \multirow{3}{*}{$B_{\theta}$} & $\{e_1,e_2\}$ &
				\multirow{3}{*}{$(\F_B\cup \{0\},\odot_B)$} & \multirow{3}{*}{$ (-1) \odot_B (-1) =   \{1,0\} $} &\\ \cline{2-2}
				& $\{e_1,a_4\}$ & & & \\ \cline{2-2}
				& $\{e_2,a_4\}$ & & & \\
				\hline \hline
				
				\multirow{3}{*}{$C(\a)_{\theta}$} & $\{e_1,e_2\}$ & $(\F_{C1}\cup \{0\},\odot_{C1})$ & $ \a \odot_{C1} \a =   \{1,\a,0\} $&\\  \cline{2-5}
				& $\{e_1,a_5\}$ & \multirow{2}{*}{$(\F_{C2}\cup \{0\},\odot_{C2})$} & \multirow{2}{*}{$ \a \odot_{C2} \a =   \{1,\a,0\} $} & \multirow{2}{*}{$\l \odot_{C2} \l = \{1,0\}$}  \\ \cline{2-2}
				& $\{e_2,a_5\}$ & & &\\
				\hline \hline
				
				$D(\b)_{\theta}$ & $\{e_1,a_6\}$ & $(\F_{D2}\cup \{0\},\odot_{D2})$ & $ \beta \odot_{D} \beta =   \{\beta,0\} $ &
				$(1-\beta) \odot_{D} (1-\beta) = 1-\beta $ \\ 
				\hline \hline
				
				\multirow{3}{*}{$E(\a,\b)_{\theta}$} & $\{e_1,e_2\}_{(\a\neq 1)}$ & $(\F_{E1}\cup \{0\},\odot_{E1})$ & $ \beta \odot_{E1} \beta =   \{1,\beta,0\} $ &
				$ \l \odot_{E1} \l  = \{1,\l, 0\} $ \\ \cline{2-5} 
				& $\{e_1,a_7\}$ & $(\F_{E2}\cup \{0\},\odot_{E2})$ & $ \a \odot_{E2} \a =   \{1,\a,0\} $ &
				$ \beta \odot_{E2} \beta  = \{1,\beta, 0\} $ \\ \cline{2-5}
				& $\{e_2,a_7\}_{(\a\neq 1)}$  & $(\F_{E3}\cup \{0\},\odot_{E3})$ & $ \a \odot_{E3} \a =   \{1,\a,0\} $ &
				$ \l \odot_{E3} \l  = \{1,\l, 0\} $ \\
				
				\hline \hline $G(\b)_{\theta}$ & $\{e_1,e_2\}$ & $(\F_{G}\cup \{0\},\odot_{G})$ & $ \beta \odot_{G} \beta =   \{1,\beta,0\} $ &
				$ 1/2 \odot_{G} 1/2  = \{1, 1/2, 0\} $ \\
				\hline \hline $H(\g)_{\theta}$ & $\{e_1,e_2\}$ & $(\F_{H}\cup \{0\},\odot_{H})$ & $ \frac{1}{2\gamma}  \odot_{H} \frac{1}{2\gamma}   =   \{1,\frac{1}{2\gamma},0\} $ &
				$  \gamma/2 \odot_{H} \gamma/2  = \{1, \gamma/2, 0\} $ \\
				\hline \hline $I_{\theta}$ & $\{a_{\a},a_{\b}\}_{\a\neq\b}$ & $(\F_{I}\cup \{0\},\odot_{I})$ & $ 1/2  \odot_{I} 1/2   =   0$ & \\
				\hline 
				\multicolumn{5}{c}{ }\\

				\caption{Axial central extensions of complex axial algebras of dimension $2$.}
				\label{t:eax2}
		\end{longtable}		}

	\end{exam}

	We highlight now that, in dimension greater than $2$, it is possible to simultaneously have $0\in \Spec(a)$  and condition~\eqref{axcond1} satisfied (cf. Corollary~\ref{cor:d2}). Indeed, consider the axial algebra $(\A,{\rm X})=(I_{\theta},\{e_1,e_2\})$, the axis $a=e_1$ and the bilinear map $\theta'\colon \A \times \A\to \mathbb{C}$ defined by $\theta'(e_2,e_3)=1$ as the only non-zero slot (note that $[\theta']\neq 0$). Then,  $(\theta')_a^{\perp}=\A$, so condition~\eqref{axcond1} is satisfied. However, $0\in \Spec(a)$, since $ae_3=0$.
	
	Note that none of the bilinear maps $\theta$ of Example~\ref{ex:eax2} is a cocycle in the sense of Definition~\ref{d:coc}. We present now an example, already known in the literature (see~\cite[Section 3.5]{Y}), to illustrate that this is not necessarily the case.
	
	\begin{exam}
		Let $F$ be a field of characteristic not $2$, and $\A$ the $4$-dimensional algebra over $F$ with basis $\{a_{-1},a_0,a_1,a_2\}$ and commutative product given by
		\begin{longtable}{lcl}
			$a_ia_i$ & $=$ & $a_i,\quad i=-1,\dots,2;$\\
			$a_ia_{i+1}$ &$=$&$2(a_i+a_{i+1})-\frac{1}{2}(a_{-1}+2a_0+2a_1+a_2),\quad i=-1,0,1;$\\
			$a_{-1}a_2$ &$=$&$\frac{1}{2}(a_{-1}+a_2);$\\
			$a_{i}a_{i+2}$ & $=$&$a_{i-1} +a_i-a_{i+1}, \quad i=-1,0,$
		\end{longtable}
		\noindent where we understand that $a_{-2}=-a_{-1}+a_1+a_2$. It is routine to check that $(\A,\{a_0,a_1\})$ is an axial algebra of Monster type $(2,\frac{1}{2})$ (i.e. regarding the fusion law $\mathcal{M}(2,\frac{1}{2})$ displayed in Table~\ref{t:m}) with eigenspaces
		
		\begin{longtable}{rcl}
			$\A_{1}^{a_0}$&$=$&$F a_0;$\\
			$\A_{0}^{a_0}$&$=$&$F (a_{-1}+2a_0-a_1-2a_2)\eqqcolon F u;$\\
			$\A_{2}^{a_0}$&$=$&$F(a_{-1}-a_1)\eqqcolon F v;$\\
			$\A_{1/2}^{a_0}$&$=$&$F(a_{-1}-a_2)\eqqcolon F w;$\\
			$\A_{1}^{a_1}$&$=$&$F a_1;$\\
			$\A_{0}^{a_1}$&$=$&$F (2a_{-1}+a_0-2a_1-a_2)\eqqcolon F u';$\\
			$\A_{2}^{a_1}$&$=$&$F(a_{0}-a_2)\eqqcolon F v';$\\
			$\A_{1/2}^{a_1}$&$=$&$F(a_{-1}-a_2)= F w.$
		\end{longtable}
		
		\begin{center}
			\begin{longtable}{|c||c|c|c|c|}
				\hline
				$\star$ & $1$ & $0$ & $2$ & $1/2$ \\
				\hline 
				\hline
				$1$ & $1$ &  & $2$ & $1/2$ \\
				\hline
				$0$ &  & $0$ & $2$ & $1/2$ \\
				\hline
				$2$ & $2$ & $2$ & $\{1,0\}$ & $1/2$ \\
				\hline
				$1/2$ & $1/2$ & $1/2$ & $1/2$ & $\{1,0,2\}$ \\
				\hline
				
				\caption{Fusion law $\mathcal{M}(2,1/2)$.}
				\label{t:m}
			\end{longtable}
			
		\end{center}

		Take a symmetric bilinear map $\theta\colon\A\times \A\to F$ with $[\theta]\neq 0$. Then, $\theta$ is a cocycle relative to $\{a_0,a_1\}$ if and only if the following equations are satisfied:
		
		\
		
		\begin{longtable}{rcl}
			$\theta(a_0,u)$&$=$&$0;$\\
			$\theta(u,v)$&$=$&$\frac{1}{2}\theta(a_0,uv);$ \\
			$\theta(u,w)$&$=$&$2\theta(a_0,uw);$ \\
			$\theta(v,w)$&$=$&$2\theta(a_0,vw);$\\
			$\theta(a_1,u')$&$=$&$0;$\\
			$\theta(u',v')$&$=$&$\frac{1}{2}\theta(a_1,u'v');$\\
			$\theta(u',w)$&$=$&$2\theta(a_1,u'w);$\\
			$\theta(v',w)$&$=$&$2\theta(a_1,v'w).$
		\end{longtable}
		
		Assuming $\theta(a_i,a_i)=0$ for $i=-1,\dots,2$, the equations above give rise to an easy system with solution
		\begin{longtable}{rclclcl}
			$\theta(a_{-1},a_1)$&$=$&$\theta(a_0,a_2)$&$=$&$0,$\\
			$\theta(a_{-1},a_0)$&$=$&$\theta(a_{-1},a_2)$&$=$&$\theta(a_0,a_1)$&$=$&$\theta(a_1,a_2).$\\
		\end{longtable}
		Take  $\theta$ in such conditions and set $\theta(a_{-1},a_0)=1$. Then $[\theta]\neq 0$, $\theta\in\Za(\A,F;\{a_0,a_1\})$  and $(\A_{\theta},\{a_0,a_1\})$ is an axial algebra of Monster type $(2,\frac{1}{2})$ by Proposition~\ref{prop:min}. In particular, $\A_{\theta}$ is the algebra ${\rm IV}_3(\frac{1}{2},2)$ of~\cite{Y}.
	\end{exam}

	The next result deals with the general case in which the set ${\rm X}$ of generating axes of $(\A,{\rm X})$ is not minimal. Recall the notation ${\rm X}^i$ introduced in the first part of this section.

	\begin{thm}\label{th:gen}
		Let $(\A,{\rm X})$ be an axial algebra, ${\rm V}$ a vector space and $\theta\colon \A\times \A\to {\rm V}$ a symmetric bilinear map such that $\{[\theta_{\g}]\}_{\g\in\Gamma}$ are linearly independent. The pair $(\A_{\theta},{\rm X}^i)$ is $(\mathcal{F}\cup\{0\},\odot^i)$-axial if and only if condition~\eqref{axcond1} holds for all $a^i_j\in {\rm X}^i$, $j=1,\dots,r^i$, for some fusion law $(\mathcal{F}\cup\{0\},\odot^i)$ containing $(\mathcal{F},\star)$.
		Furthermore, we can take $\odot^i=\star$ if and only if $\theta\in\Za(\A,{\rm V};{\rm X}^i)$.
		
		Also, define the set
		\[{\rm Y}=\{a+\theta(a,a)\mid a\in {\rm X}\}.\] 
		For every ${\rm Y}'\subseteq {\rm Y}$ such that there exists $i\in I$ with ${\rm X}^i\subseteq P({\rm Y}')$, $(\A_{\theta},{\rm Y}')$ is $(\mathcal{F}\cup\{0\},\odot)$-axial for some fusion law $(\mathcal{F}\cup\{0\},\odot)$ containing $(\mathcal{F},\star)$ if and only if condition~\eqref{axcond1} is satisfied for all $a\in P({\rm Y'})$. We can take $\odot=\star$ if and only if $\theta\in\Za(\A,{\rm V};P({\rm Y}'))$.
	\end{thm}
	
	\begin{proof}
		The first part follows directly from Proposition~\ref{prop:min}.
		
		Regarding the second part, it is clear that ${\rm Y}$ is composed of idempotent elements. By Lemma~\ref{l:gen}, the set $\{a^i_j+\theta(a^i_j,a^i_j)\}_{j=1}^{r^i}$ generates $\A_{\theta}$,  and consequently, so does ${\rm Y}'$.
		Finally, given $\lambda,\mu\in\mathcal{F}$, set $\lambda\odot \mu=\lambda\star\mu$ if $\theta$ is a cocycle relative to $P({\rm Y}')$ or $\lambda\odot \mu=\lambda\star\mu\cup\{0\}$ otherwise, and $0\odot 0=0\odot \lambda=\emptyset$ for all $\lambda\in \mathcal{F}$ if $0\notin\mathcal{F}$. Using Proposition~\ref{prop:ss}, it is obvious that the fusion law $(\mathcal{F}\cup\{0\},\odot)$ defined in this way satisfies the conditions of the theorem.
	\end{proof}

	Note that, contrary to the situation described after Proposition~\ref{prop:min}, if $(\mathcal{F},\star)$ is a minimal fusion law for $(\A,{\rm X})$ and $(\A_{\theta},{\rm Y})$ is axial, with the notation of Theorem~\ref{th:gen}, $(\mathcal{F}\cup\{0\},\odot)$ does not need to be minimal for $(\A_{\theta},{\rm Y})$.

	Assume that the fusion law $(\F,\star)$ is $\zd$-graded and that the ground field has characteristic not $2$. It would be interesting to know if, given $\theta\colon \A\times \A \to {\rm V}$ and a certain set of axes ${\rm X}'\subseteq {\rm X}$, the condition $\theta\in\Za(\A,{\rm V};{\rm X}')$ is stable in the sense of~\cite{axial}; i.e. given another set of axes ${\rm X}''\subseteq {\rm X}$ such that $\overline{{\rm X}'}=\overline{{\rm X}''}$, than $\theta\in\Za(\A,{\rm V};{\rm X}'')$. 
	\begin{prop}\label{prop:stab}
		Let $(\A,{\rm X})$ be an axial algebra over a field of characteristic not $2$, ${\rm X}'\subseteq {\rm X}$ a subset, ${\rm V}$ a vector space and $\theta\colon \A\times \A\to {\rm V}$ a symmetric bilinear map such that $\{[\theta_{\g}]\}_{\g\in\Gamma}$ are linearly independent. Furthermore, assume that $(\F,\star)$ admits a $\zd$-grading. Then, the condition $\theta\in\Za(\A,{\rm V};{\rm X}')$ is stable.
	\end{prop}
	\begin{proof}
		As the generation of axial algebras is stable~\cite[Theorem 3.12]{axial}, it suffices to see that every element $c$ in the closure $\overline{{\rm X}'}$ satisfies that $c+\theta(c,c)$ is an $(\F\cup\{0\},\star)$-axis for $\A_{\theta}$. Let $a,b\in {\rm X}'$, and set $c=\phi_b(a)\in\overline{{\rm X}'}$. Note that $(\F\cup\{0\},\star)$ admits the same $\zd$-grading as $(\F,\star)$ (if $0\notin \F$, it suffices with including $0$ in any of the graded components), so we can construct its Miyamoto group. By hypothesis, $a+\theta(a,a)$ and $b+\theta(b,b)$ are $(\mathcal{F}\cup\{0\},\star)$-axes of  $\A_{\theta}$, and so is $\phi_{b+\theta(b,b)}(a+\theta(a,a))$ by~\cite[Section 3.2]{axial}. We claim that $\phi_{b+\theta(b,b)}(a+\theta(a,a))=c+\theta(c,c)$. Indeed, 
		consider the eigenspace decompositions of $a$ in $\A$ and $\A_{\theta}$ regarding $b$ and $b+\theta(b,b)$, respectively: \begin{align} \label{eq:stab}
			a=&\sum_{\lambda\in \F_{1},\l\neq 0}a_{\l}^b + \sum_{\lambda\in \F_{-1},\l\neq 0}a_{\l}^b + a_0^b \notag \\ =&\sum_{\lambda\in \F_{1},\l\neq 0}\left(a_{\l}^b+\l^{-1}\theta\left(b,a^b_{\l}\right)\right)+\sum_{\l\in \F_{-1},\l\neq 0}\left(a_{\l}^b+\l^{-1}\theta\left(b,a^b_{\l}\right)\right)+a_0^{b+\theta(b,b)},
		\end{align} where $a^b_{\l}\in \A^b_{\l}$ and $a_0^{b+\theta(b,b)}\in \left(\A_{\theta}\right)^{b+\theta(b,b)}_0$. Also, note that $\theta(a,a)\in \left(\A_{\theta}\right)^{b+\theta(b,b)}_0$, and define $\sigma(0)=1$ if $0\in(\F\cup\{0\})_{1}$ and $\sigma(0)=-1$ if $0\in(\F\cup\{0\})_{-1}$. Then, 
		\begin{align*}
			\phi_{b+\theta(b,b)}(a+\theta(a,a))=&\sum_{\lambda\in \F_{1},\l\neq 0}\left(a_{\l}^b+\l^{-1}\theta\left(b,a^b_{\l}\right)\right)-\sum_{\l\in \F_{-1},\l\neq 0}\left(a_{\l}^b+\l^{-1}\theta\left(b,a^b_{\l}\right)\right) \\ &+\sigma(0) \left(a_0^{b+\theta(b,b)} + \theta(a,a)\right) \\ =& c + \sum_{\lambda\in \F_{1},\l\neq 0}\l^{-1}\theta\left(b,a^b_{\l}\right)-\sum_{\l\in \F_{-1},\l\neq 0}\l^{-1}\theta\left(b,a^b_{\l}\right) \\ & + \sigma(0) \left(a_0^{b+\theta(b,b)} + \theta(a,a)- a_0^b\right) \\ \overset{\eqref{eq:stab}}{=}&  c + \sum_{\lambda\in \F_{1},\l\neq 0}\l^{-1}\theta\left(b,a^b_{\l}\right)-\sum_{\l\in \F_{-1},\l\neq 0}\l^{-1}\theta\left(b,a^b_{\l}\right) \\ & + \sigma(0) \left( \theta(a,a)- \sum_{\lambda\in \F_{1},\l\neq 0}\l^{-1}\theta\left(b,a^b_{\l}\right)-\sum_{\l\in \F_{-1},\l\neq 0}\l^{-1}\theta\left(b,a^b_{\l}\right)\right) \\ \eqqcolon & c + v,
		\end{align*}
		where $v\in V$.
		As $c+v$ is an idempotent element of $\A_{\theta}$, it follows
		\[c+v=[c+v,c+v]_{\theta}=c+\theta(c,c),\]
		and the claim is proved. 
	\end{proof}
	
	Recall that the isomorphy class of a central extension $\A_{\theta}$ does not depend on the particular choice of the point in the orbit of $[\theta]$ under the natural action of $\Aut(A)$. Then, it is apparent that if $\theta\in \Za(\A,{\rm V};{\rm X})$, it follows that $\phi\theta\in \Za(\A,{\rm V};\phi^{-1}({\rm X}))$.
	But if further $\phi\in G({\rm X})$ and the conditions of Proposition~\ref{prop:stab} are satisfied, then we also have the following stronger result.
	
	\begin{cor}
		Let $(\A,{\rm X})$ be an axial algebra over a field of characteristic not $2$, ${\rm V}$ a vector space and $\theta\colon \A\times \A\to {\rm V}$ a symmetric bilinear map such that $\{[\theta_{\g}]\}_{\g\in\Gamma}$ are linearly independent. Furthermore, assume that $(\F,\star)$ admits a $\zd$-grading. Then, there exists a natural action of $G({\rm X})$ on $\Za(\A,{\rm V};{\rm X})$.
	\end{cor}
	
	\begin{proof}
		Let $\theta\in \Za(\A,{\rm V};{\rm X})$ and $\phi\in G({\rm X})$; we are going to prove that $\phi\theta\in \Za(\A,{\rm V};{\rm X})$. Take $a\in {\rm X}$. If $x\in \ker L_a$, clearly $\phi(a)\phi(x)=0$. By Proposition~\ref{prop:stab},  $\phi\theta(a,x)=\theta(\phi(a),\phi(x))=0$ and condition~\eqref{axcond1} is satisfied. Now take $x\in \A^a_{\l}$ and $y\in \A^a_{\mu}$ for some $\l,\mu\in\F$ such that $0\notin\l\star\mu$. A straightforward computation shows that $\A_{\nu}^{\phi(a)}=\phi(\A_{\nu}^a)$ for all $\nu \in \F$, so $\phi(x)\in \A^{\phi(a)}_{\l}$ and $\phi(y)\in \A^{\phi(a)}_{\mu}$. Then Proposition~\ref{prop:stab} ensures that \[\phi\theta(x,y)=\theta(\phi(x),\phi(y))=\sum_{\nu\in\l\star\mu}\nu^{-1}\theta(\phi(a),z_{\nu}^{\phi}),\] for some $z_{\nu}^{\phi}\in\A^{\phi(a)}_{\nu}$ such that $\phi(x)\phi(y)=\sum_{\nu\in\l\star \mu}z_{\nu}^{\phi}$. Noting that \[\phi(x)\phi(y)=\phi(xy)=\sum_{\nu\in\l\star \mu}\phi(z_{\nu}),\] where the $z_{\nu}\in \A_{\nu}^a$ stand for their usual meaning, we see that $z_{\nu}^{\phi}=\phi(z_{\nu})$ for all $\nu\in\l\star \mu$. Thus, condition~\eqref{axcond2} is also satisfied and $\phi\theta\in \Za(\A,{\rm V};{\rm X})$.
	\end{proof}
	
	Now, we present an important result that justifies the importance of studying central extensions of axial algebras.

	\begin{thm}\label{th:imp}
		Let  $(\Ba,{\rm Y})$ be an $(\mathcal{F},\star)$-axial algebra with non-zero annihilator. Then, there exists another $(\mathcal{F},\star)$-axial algebra $(\A,{\rm X})$ and a cocycle $\theta\in\Za(\A,\Ann(\Ba);{\rm X})$ such that $\Ba= \A_{\theta}$. Also, if ${\rm Y}$ is a minimal generating set of axes for $\Ba$, ${\rm X}$ is a minimal generating set of axes for $\A$.
	\end{thm}
	
	\begin{proof}
		Take a linear complement $\A$ of $\Ann(\Ba)$ and set $P_{\A}\colon \Ba\to \A$ defined by $P_{\A}(x+v)=x$, with $x\in\A$ and $v\in \Ann(\Ba)$. We endow $\A$ with the product $xy=P_{\A}([x,y])$, where $[,]$ denotes the product in $\Ba$, in order to give it structure of algebra. Note that, with this structure, $P_{\A}$ is a homomorphism of algebras: indeed, for all $x+v,y+w\in \Ba$,
		\[P_{\A}([x+v,y+w])=P_{\A}([P_{\A}(x+v),P_{\A}(y+w)])=P_{\A}(x+v)P_{\A}(y+w).\]
		
		Set ${\rm X}=P_{\A}({\rm Y})$. Since ${\rm Y}$ is a generating set for $\Ba$, ${\rm X}$ generates $P_{\A}(\Ba)=\A$. Take $a\in {\rm X}$ such that $a=P_{\A}(b)$ for a certain $b\in {\rm Y}$.  Then \[aa=P_{\A}(b)P_{\A}(b)=P_{\A}([b,b])=P_{\A}(b)=a,\] and $a$ is idempotent. 
		
		Write $L^{\Ba}_b$ for the left multiplication by $b$ operator in $\Ba$, and $\Spec_{\Ba}(b)$ for its spectrum. We reserve the notations $L_a$ and $\Spec(a)$ for their correspondences in $\A$. It is clear that $0\in\Spec_{\Ba}(b)$ and $\Ann(\Ba)\subseteq \Ba_0^b$. Choose a basis $\{z_{\beta}^b\}_{\b \in\mathcal{B}}$ of $\Ann(\Ba)$ and complete it to a basis $\{z_{\b}^b\}_{\b \in\mathcal{B}'}$ of $\Ba$ formed by eigenvectors of $L^{\Ba}_b$, with $[b,z_{\b}^b]=\lambda_{\b}z_{\b}^b$. The elements $\{P(z_{\a}^b)\}_{\a \in\mathcal{B}'\setminus\mathcal{B}}$ form a basis for $\A$, and are in fact eigenvectors of $ L_a$ with respect to $\lambda_{\b}$.
		Note that $\Spec(a)=\Spec_{\Ba}(b)$ if and only if $ \Ba_0^b \neq \Ann (\Ba)$; otherwise, $\Spec(a)=\Spec_{\Ba}(b)\setminus\{0\}$.
		
		The above explanations show that, for every $\lambda\in\Spec(a)$, $\A_{\lambda}^a=P_{\A}(\Ba_{\lambda}^b)$, and therefore
		\begin{align*}
			\A_{\lambda}^a\A_{\mu}^a=P_{\A}(\Ba_{\lambda}^b)P_{\A}(\Ba_{\mu}^b)=P_{\A}[\Ba_{\lambda}^b,\Ba_{\mu}^b]\subseteq P_{\A}(\Ba_{\lambda\star\mu}^b)=\A_{\lambda\star\mu}^a,
		\end{align*}
		for all $\lambda,\mu\in\Spec(a)$, and assuming that $\A_0^a=\{0\}$ if $0\notin\Spec(a)$, as always.
		
		Summing up, we have proved that $(\A,{\rm X})$ is an $(\mathcal{F},\star)$-axial algebra.
		
		Now, define $\theta\colon\A\times\A\to\Ann(\Ba)$ by $\theta(x,y)=[x,y]-xy$, and construct $\A_{\theta}$ in the usual way. For $x+v, y+w\in\A_{\theta}$, we have that
		\[[x+v,y+w]_{\theta}=xy+\theta(x,y)=[x,y],\]
		so $\A_{\theta}=B$, and $(\A_{\theta},{\rm Y})$ is $(\mathcal{F},\star)$-axial. Proposition~\ref{prop:ss} yields that $\theta\in\Za(\A,\Ann(B);{\rm X})$.
		
		Finally, assume that ${\rm Y}$ is a minimal generating set of axes for $\Ba$. Take a minimal generating set of axes for $\A$, ${\rm X}'\subseteq {\rm X}$. Then, ${\rm X}'$ would also generate $\A_{\theta}=\Ba$ by Lemma~\ref{l:gen}. It follows that also ${\rm Y}'=\{b\in {\rm Y}\mid P(b)\in {\rm X}'\}\subset {\rm Y}$ would generate $\Ba$, a contradiction.
	\end{proof}
	
	\begin{rem}
		Note that, in the conditions of Theorem~\ref{th:imp}, in some cases we can find another fusion law $(\mathcal{G},\odot)\subseteq(\mathcal{F},\star)$ such that $(\A,{\rm X})$ is $(\mathcal{G},\odot)$-axial. Namely, if $\Ba_0^b = \Ann(\Ba)$ for all $b\in {\rm Y}$, we can set $\mathcal{G}=\mathcal{F}\setminus \{0\}$ and $\lambda\odot \mu=\lambda\star\mu\setminus\{0\}$ for all $\lambda,\mu\in\mathcal{G}$. On the contrary, assume that there exists $b\in B$ such that $\Ba_0^b \neq \Ann(\Ba)$. Then, we must set $\mathcal{G}=\mathcal{F}$ and, given $\lambda,\mu\in\Spec_{\Ba}(b)=\Spec(a)$, we can set $\lambda\odot\mu=\lambda\star\mu\setminus\{0\}$ if and only if $\Pi^b_0([\Ba_{\lambda}^b,\Ba_{\mu}])\subseteq \Ann(\Ba)$, where $\Pi_0^b\colon \Ba\to \Ba_0^b$ is the natural projection; otherwise, we must set $\lambda\odot\mu=\lambda\star\mu$.
	\end{rem}
	
	The following corollary is a direct consequence of Theorem~\ref{th:imp}.
	
	\begin{cor}
		Let $\A$ be an algebra, ${\rm V}$ a vector space and $\theta\colon \A\times\A\to {\rm V}$ a bilinear form. If $(\A_{\theta},{\rm Y})$ is axial with respect to some fusion law $(\mathcal{F},\star)$, then $(\A,P({\rm Y}))$ is also $(\mathcal{F},\star)$-axial.
	\end{cor}
	
	We finish this section relating some properties of an axial algebra $(\A,{\rm X})$ with those of its central extensions. First, we provide an easy lemma whose proof is left to the reader.
	
	\begin{lm}\label{lm:frob}
		Let $(\A,{\rm X})$ be an axial algebra admitting a Frobenius form $(\cdot,\cdot)$. Then, $\Ann(\A)$ is contained in the radical of $(\cdot,\cdot)$.
	\end{lm}
	
	\begin{prop}\label{p:prop}
		Let $(\A,{\rm X})$ be an $(\F,\star)$-axial algebra admitting an $(\F\cup\{0\},\odot)$-axial central extension $(\A_{\theta},{\rm Y})$, where ${\rm Y}=\{a+\theta(a,a)\mid a \in {\rm X}\}$. Then:
		\begin{enumerate}
			\item\label{prim} $(\A,{\rm X})$ is primitive if and only if $(\A_{\theta},{\rm Y})$ is primitive.
			\item $(\A,{\rm X})$ admits a Frobenius form if and only if $(\A_{\theta},{\rm Y})$ admits a Frobenius form.
			\item The radical of $(\A_{\theta}, {\rm Y})$ is $R(\A_{\theta},{\rm Y})=R(\A,{\rm X})\dot{+} {\rm V}$. Conversely, the radical of $(\A,{\rm X})$ is $R(\A,{\rm X})=P(R(\A_{\theta},{\rm Y}))$. 
			\item If $(\A,{\rm X})$ is $m$-closed for a certain $m\in\N$, then $(\A_{\theta},{\rm Y})$ is at most $(m+1)$-closed.
			\item\label{flip} If $(\A,{\rm X})$ is symmetric with flip $\tau$ and there exists an automorphism $\varphi\in\Aut({\rm V})$ such that $\varphi(\theta(x,y))=\theta (\tau(x),\tau(y))$ for all $x,y\in \A$, then $(\A_{\theta},{\rm Y})$ is symmetric with flip $\tau_{\theta}$ defined by $\tau_{\theta}(x+v)=\tau(x)+\varphi (v)$. 
			\item \label{miy} Assume that the characteristic of the ground field is different from $2$ and that both $(\F,\star)$ and $(\F\cup\{0\},\odot)$ admit $\zd$-gradings such that the one of $(\F\cup\{0\},\odot)$ contains that of $(\F,\star)$. Then, the axial subgroups $T_a=\{\id_{\A},\phi_a\}\subseteq \Aut(\A)$ and $T_{a+\theta(a,a)}=\{\id_{\A_{\theta}},\phi_{a+\theta(a,a)}\}\subseteq \Aut (\A_{\theta})$ are isomorphic for all $a\in {\rm X}$. 
			Moreover, if the assignment $\phi_a \mapsto \phi_{a+\theta(a,a)}$ gives rise to a homomorphism between the Miyamoto groups $G({\rm X})$ and $G({\rm Y})$, then it is bijective and $G({\rm X})\simeq G({\rm Y})$.
		\end{enumerate}
	\end{prop}
	
	\begin{proof}
		\hfill
		\begin{enumerate}
			\item It follows from Lemma~\ref{l:eig}.
			\item 	Let $(\cdot,\cdot)$ be a Frobenius form for $\A$. Then,  $(\cdot,\cdot)_{\theta}\colon \A_{\theta}\times\A_{\theta}\to F$ defined by $(x+v,y+w)_{\theta}=(x,y)$ is a Frobenius form for $\A_{\theta}$.
			
			Conversely, given a Frobenius form $(\cdot,\cdot)_{\theta}$ for $\A_{\theta}$, define the bilinear form $(\cdot,\cdot)\colon \A\times\A\to F$ by $(x,y)=(x,y)_{\theta}$. By Lemma~\ref{lm:frob},
			\[(x,yz)=(x,yz)_{\theta}=(x,[y,z]_{\theta})_{\theta}=([x,y]_{\theta},z)_{\theta}=(xy,z)_{\theta}=(xy,z)\] for all $x,y\in\A$,
			so $(\cdot,\cdot)$ is a Frobenius form for $\A$.
			\item It follows from the definition of radical.
			\item It follows from the proof of Lemma~\ref{l:gen}.
			\item Routine. 
			\item Routine.
		\end{enumerate}
	\end{proof}
	
	\begin{exam}
		Proposition~\ref{p:prop} and Example~\ref{ex:ax2} allow us to obtain some basic properties of the algebras in Example~\ref{ex:eax2}. Note that the additional condition of~\ref{p:prop}\eqref{flip} only holds for $(B,\{e_1,e_2\})$ and $(I,\{e_1,e_2\})$; in both cases, it suffices to take $\varphi=\id_{\mathbb{C}}$. However, this condition is not necessary: the map on $(B_{\theta},\{e_1,a_4\})$ defined by $\tau(e_1)=a_4$, $\tau(e_2)=e_2$ and $\tau(e_3)=-e_3$ is indeed a flip.
		
		On the other hand, $(D(\a),\{e_i,a_6\})$ for $i=1,2$, and $(B,{\rm X})$ and $(I,{\rm X})$ for all choices of ${\rm X}$ satisfy the additional conditions in~\ref{p:prop}\eqref{miy}. Therefore, the Miyamoto groups of $(D(\a)_{\theta},\{e_i,a_6\})$ for $i=1,2$, $(B_{\theta},{\rm X})$ and $(I_{\theta},{\rm X})$ are isomorphic to $\zd$, $S_3$ and $\mathfrak{F}_2$, respectively.
	\end{exam}

	\section{Axial central extensions of Jordan algebras}
	\label{s:jordan}
	Alongside this section, we will assume that the characteristic of the ground fields is different from $2$.
	
	One of the most well-known features of the variety of Jordan algebras is the Peirce decomposition. This can be naturally expressed in the language of axial algebras: every idempotent of a Jordan algebra is an $\mathcal{J}(\frac{1}{2})$-axis, where $\mathcal{J}(\frac{1}{2})$ is the fusion law displayed in Table~\ref{table:j12}.

	\begin{longtable}{|c||c|c|c|}
		\hline
		$\star$ & $1$ & $0$ & $1/2$ \\
		\hline 
		\hline
		$1$ & $1$ &  &$1/2$ \\
		\hline
		$0$ &  & $0$  & $1/2$ \\
		\hline
		$1/2$ & $1/2$ & $1/2$ & $\{1,0\}$ \\
		\hline
		
		\caption{Fusion law $\mathcal{J}(\frac{1}{2})$.}
		\label{table:j12}
	\end{longtable}
	
	Then, every Jordan algebra generated by its idempotent elements is $\mathcal{J}(\frac{1}{2})$-axial. What is more, it is an open question~\cite[Conjecture 4.3]{ms23} to know if Jordan algebras generated by idempotents, together with quotients of Matsuo algebras, are the only axial algebras of Jordan type $\mathcal{J}(\frac{1}{2})$; the answer has been proved affirmative for  algebras 
	generated by two or three primitive axes~\cite{hrs2,gs}. Studying central extensions of Jordan algebras could be a way to delve into this question.

	It turns out (cf.~\cite{a46,a47}) that every finite-dimensional simple Jordan algebra over $\mathbb{C}$ is generated by idempotents, so we can apply the results of Section~\ref{s:ce} to study which of their non-split central extensions are $\mathcal{J}(\frac{1}{2})$-axial. The first aim of this section is to prove the following theorem.
	
	\begin{thm}\label{th:j}
		Let $J$ be a finite-dimensional simple Jordan algebra over $\mathbb{C}$, and let $\rm X$ be the set of all its idempotents. There do not exist  non-split central extensions $J_{\theta}$ of $J$ which are $\mathcal{J}(\frac{1}{2})$-axial with respect to the set ${\rm Y}=\{a+\theta(a,a)\mid a\in {\rm X}\}$.
	\end{thm}
	
	The significance of Theorem~\ref{th:j} lies on the fact that it generalises the non-existence of non-split central extensions in the variety of Jordan algebras of such an algebra $J$~\cite{glass}. 
	
	To prove it, we will rely on the classification of the finite-dimensional simple Jordan algebras over $\mathbb{C}$~\cite{a46,a47}.

	\begin{itemize}
		
		\item Type $\mathfrak{A}$: algebras of complex $n\times n$-matrices $\mathcal{M}_n(\mathbb{C})$, with product \begin{equation}\label{jprod}
			XY=\frac{1}{2}(X\circ Y + Y\circ X),
		\end{equation}
		where $\circ$ denotes the usual product of matrices.
		\item Type $\mathfrak{B}$: algebras of complex symmetric  $n\times n$-matrices $\textit{Sym}_n(\mathbb{C})$, with product given by~\eqref{jprod}.
		\item Type $\mathfrak{C}$:  algebras of complex  $J_n$-symmetric  $n\times n$-matrices \[\textit{Sym}_n(J,\mathbb{C})=\{X\in\mathcal{M}_n(\mathbb{C})\mid J_n^{-1} X^T J_n=X \},\] where\[J_n=\left(\begin{array}{cc}
			0   & \id_n \\
			-\id_n    & 0
		\end{array}\right),\] with product given by~\eqref{jprod}.
		\item Type $\mathfrak{D}$:  algebras with underlying vector space $\mathbb{C}^n$ and product given by \[xy=(x^Te_n)y+(y^Te_n)x-(x^Ty)e_n,\] where $\{e_i\}_{i=1}^n$ is the canonical basis of $\mathbb{C}^n$.
		\item Type $\mathfrak{E}$: the algebra of $3\times 3$-hermitian matrices over $\mathbb{O}_{\mathbb{C}}$ \[\textit{Herm}_3(\mathbb{O}_{\mathbb{C}})=\{X\in\mathcal{M}_3(\mathbb{O}_{\mathbb{C}})\mid X^T=X^* \},\]
		where $X^*$ means the conjugate matrix of $X$, with product given by~\eqref{jprod}. 
	\end{itemize}

	\begin{proof}
		
		We will deal with each type of the classification separately. Recall that, in any case, if $e$ is an idempotent in $\A$, then $e+\theta(e,e)$ is an idempotent in $\A_{\theta}$.
		
		\begin{itemize}
			\item Type $\mathfrak{A}$.
			
			Let $\A$ be an algebra of type $\mathfrak{A}$, ${\rm V}$ a vector space over $\mathbb{C}$ and $\theta\colon \A\times \A \to {\rm V}$ a bilinear map. Consider the idempotents $a_i=E_{ii}$ for $i=1,\dots,n$, with Peirce decompositions
			\begin{longtable}{lcl}
				$\A_{1}^{a_i}$&$=$& $\mathbb{C} a_i;$\\
				$\A_{0}^{a_i}$&$=$&$\lsp\{E_{jk}\mid j,k=1,\dots,n,\ j,k\neq i\};$\\
				$\A_{1/2}^{a_i}$&$=$&$\lsp\{E_{ij}, E_{ki}\mid j,k=1,\dots,n,\ j,k\neq i\},$
			\end{longtable}
			\noindent and assume, without loss of generality (see equation~\eqref{eq:thid0}), that $\theta(a_i,a_i)=0$ for all $i=1,\dots,n$.  We apply Proposition~\ref{prop:ss} to determine what values $\theta$ must have in order that the idempotents $a_i$ are in fact $\mathcal{J}(\frac{1}{2})$-axes in $\A_{\theta}$. From condition~\eqref{axcond1}, it follows that \[\theta(E_{ii},E_{jk})=0\] for $j,k\neq i$, and from condition~\eqref{axcond2}, we obtain that 
			\begin{longtable}{rcl}
				$\theta(E_{ij},E_{kl})$&$=$&$0, \quad j,l\neq i,\ j\neq k,\ (i,j)\neq (k,l);$\\
				$\theta(E_{ij},E_{jk})$&$=$&$\theta(E_{ii},E_{ik}), \quad j,k\neq i.$
			\end{longtable}
			
			We consider also the idempotents $a_{ij}=E_{ii}+E_{ij}$ for $i,j=1,\dots,n$, $i\neq j$, with eigenspace decomposition
			\begin{longtable}{lcl}
				$\A_{1}^{a_{ij}}$&$=$&$\mathbb{C} a_{ij};$\\
				$\A_{0}^{a_{ij}}$&$=$&$\lsp\{E_{ik}-E_{jk},E_{lk}\mid k,l=1,\dots,n,\ k,l\neq i,\ l\neq j\};$\\
				$\A_{1/2}^{a_{ij}}$&$=$&$\lsp\{E_{ii}-E_{jj}-E_{ji},E_{ik}, E_{li}+E_{lj}\mid k,l=1,\dots,n,\ k,l\neq i,\ l\neq j\},$
			\end{longtable}
			\noindent and study what values $\theta$ must take so that the idempotents $a_{ij}+\theta(a_{ij},a_{ij})$ are $\mathcal{J}(\frac{1}{2})$-axes in $\A_{\theta}$. We obtain from condition~\eqref{axcond1} that \[\theta(E_{ij},E_{ij})=0,\] and from condition~\eqref{axcond2}, that \[\theta(E_{ij},E_{ji})=0,\] for $j\neq i$.
			Finally, it is easy to check that $\theta=\delta f$ for
			\begin{longtable}{rcl}
				$f\colon \A$ & $\to$&$ {\rm V},$ \\
				$E_{ij}$ & $\mapsto$&$ \begin{cases} 2\theta(E_{ii},E_{ij}),\quad \textup{if } i\neq j; \\ 0,\quad \textup{if } i=j,\end{cases}$
			\end{longtable}
			\noindent and therefore $[\theta]=0$.
			
			\item Type $\mathfrak{B}$.
			
			Let $\A$ be an algebra of type $\mathfrak{B}$, ${\rm V}$ a vector space over $\mathbb{C}$ and $\theta\colon \A\times \A \to {\rm V}$ a bilinear map. Consider the idempotents $a_i=E_{ii}$ for $i=1,\dots,n$, with Peirce decompositions
			\begin{longtable}{rcl}
				$\A_{1}^{a_i}$&$=$& $\mathbb{C} a_i;$\\
				$\A_{0}^{a_i}$&$=$&$\lsp\{E_{jk}+E_{kj}\mid j,k\neq i\};$\\
				$\A_{1/2}^{a_i}$&$=$&$\lsp\{E_{ij}+ E_{ji}\mid j\neq i\},$
			\end{longtable}
			\noindent and assume, without loss of generality (see equation~\eqref{eq:thid0}), that $\theta(a_i,a_i)=0$ for all $i=1,\dots,n$. Applying Proposition~\ref{prop:ss}, we  determine what values $\theta$ must have in order that the idempotents $a_i$ are $\mathcal{J}(\frac{1}{2})$-axes also in $\A_{\theta}$. From condition~\eqref{axcond1},
			\begin{longtable}{rcl}
				$\theta(E_{ii},E_{jj})$&$=$&$0;$\\
				$\theta(E_{ii},E_{jk}+E_{kj})$&$=$&$0,$\\
			\end{longtable}
			\noindent for $j,k\neq i$, $j\neq k$, and from condition~\eqref{axcond2}, 
			\begin{longtable}{rcl}
				$\theta(E_{ij}+E_{ji},E_{kl}+E_{lk})$&$=$&$0;$\\
				$\theta(E_{ij}+E_{ji},E_{jk}+E_{kj})$&$=$&$\theta(E_{ii},E_{ik}+E_{ki}),$ \\
			\end{longtable}
			\noindent where the indexes $i,j,k,l$ must take different values.
			
			Now, we take into account the idempotents $a_{ij}=\frac{1}{2}(E_{ii}+E_{jj}+E_{ij}+E_{ji})$ for $i,j=1,\dots,n$, $i\neq j$, with eigenspace decomposition
			\begin{longtable}{lcl}
				$\A_{1}^{a_{ij}}$&$=$&$\mathbb{C} a_{ij};$\\
				$\A_{0}^{a_{ij}}$&$=$&$\lsp\{E_{ii}+E_{jj}-E_{ij}-E_{ji},E_{ik}+E_{ki}-E_{jk}-E_{jk},E_{kl}+E_{lk}\mid  k,l\neq i,j\};$\\
				$\A_{1/2}^{a_{ij}}$&$=$&$\lsp\{E_{ii}-E_{jj},E_{ik}+E_{ki}+E_{jk}+E_{jk}\mid k\neq i,j\}.$\\
			\end{longtable}
			\noindent By condition~\eqref{axcond1}, we obtain that \[\theta(E_{ij}+E_{ji},E_{ij}+E_{ji})=0\] is a necessary condition so that 
			$a_{ij}+\theta(a_{ij},a_{ij})$ is semisimple in $\A_{\theta}$ and follows the fusion law $\mathcal{J}(\frac{1}{2})$. Then $\theta=\delta f$ for
			\begin{longtable}{rcl}
				$f\colon \A$ & $\to$&$ {\rm V},$ \\
				$E_{ij}+E_{ji}$ & $\longmapsto$&$ \begin{cases} 2\theta(E_{ii},E_{ij}+E_{ji}),\quad \textup{if } i\neq j; \\ 0,\quad \textup{if } i=j,\end{cases}$
			\end{longtable}
			\noindent so $[\theta]=0$.
			
			\item Type $\mathfrak{C}$.
			
			Let $\A$ be an algebra of type $\mathfrak{C}$, ${\rm V}$ a vector space over $\mathbb{C}$ and $\theta\colon \A\times \A \to {\rm V}$ a bilinear map. Note that $\{E_{ij}+E_{(n+j)(n+i)},E_{i(n+j)}-E_{j(n+i)},E_{(n+i)j}-E_{(n+j)i}\}_{i,j=1}^n$ is a basis of $\A$. As we have done in the previous cases, we will obtain conditions on $\theta$ necessary to preserve the semisimplicity of the idempotents of $\A$ in $\A_{\theta}$ and the fusion law $\mathcal{J}(\frac{1}{2})$. 
			
			Consider first the idempotents $a_i=E_{ii}+E_{(n+i)(n+i)}$, with Peirce decomposition
			\begin{longtable}{lcl}
				$\A_{1}^{a_i}$&$=$& $\mathbb{C} a_i;$\\
				$\A_{0}^{a_i}$&$=$&$\lsp\{E_{jk}+E_{(n+k)(n+j)},E_{j(n+k)}-E_{k(n+j)},E_{(n+j)k}-E_{(n+k)j}\mid j,k\neq i\};$\\
				$\A_{1/2}^{a_i}$&$=$&$\lsp\{E_{ij}+ E_{(n+j)(n+i)},E_{ji}+ E_{(n+i)(n+j)},E_{i(n+j)}-E_{j(n+i)},E_{(n+i)j}-E_{(n+j)i}\mid j\neq i\}.$ \end{longtable}
			Following the same steps as for algebras of type $\mathfrak{A}$, we obtain 
			\begin{longtable}{lcl}
				$\theta(E_{ii}+E_{(n+i)(n+i)},E_{ii}+E_{(n+i)(n+i)})$&$=$&$0;$\\
				$\theta(E_{ii}+E_{(n+i)(n+i)},E_{jk}+E_{(n+k)(n+j)})$&$=$&$0, \quad  j,k\neq i;$\\
				$\theta(E_{ij}+E_{(n+j)(n+i)},E_{kl}+E_{(n+l)(n+k)})$&$=$&$0, \quad j,l\neq i,\ j\neq k,\ (i,j)\neq (k,l);$\\
				$\theta(E_{ij}+E_{(n+j)(n+i)},E_{jk}+E_{(n+k)(n+j)})$&$=$&$\theta(E_{ii}+E_{(n+i)(n+i)},E_{ik}+E_{(n+k)(n+i)}), \quad j,k\neq i;$\\
				$\theta(E_{ij}+E_{(n+j)(n+i)},E_{ij}+E_{(n+j)(n+i)})$&$=$&$0, \quad j\neq i;$\\
				$\theta(E_{ij}+E_{(n+j)(n+i)},E_{ji}+E_{(n+i)(n+j)})$&$=$&$0, \quad j\neq i.$\\
			\end{longtable}
			\noindent But more information can be obtained from these idempotents. After a careful inspection and the application of conditions~\eqref{axcond1} and~\eqref{axcond2}, respectively, we obtain that 
			\begin{longtable}{lcl}
				$\theta(E_{ii}+E_{(n+i)(n+i)},E_{j(n+k)}-E_{k(n+j)})$&$=$&$0, \quad  j,k\neq i,\ j\neq k;$\\
				$\theta(E_{ii}+E_{(n+i)(n+i)},E_{(n+j)k}-E_{(n+k)j})$&$=$&$0, \quad  j,k\neq i,\ j\neq k;$
			\end{longtable}
			\noindent and
			\begin{longtable}{lcl}
				$\theta(E_{ij}+ E_{(n+j)(n+i)},E_{k(n+l)}-E_{l(n+k)})$&$=$&$0, \quad  k,l\neq i,j,\ k\neq l;$\\
				$\theta(E_{ij}+ E_{(n+j)(n+i)},E_{(n+k)l}-E_{(n+l)k})$&$=$&$0, \quad  k,l\neq i,j,\ k\neq l;$\\
				$\theta(E_{i(n+j)}-E_{j(n+i)},E_{k(n+l)}-E_{l(n+k)})$&$=$&$0, \quad  j,k,l\neq i,\ k\neq l;$\\
				$\theta(E_{(n+i)j}-E_{(n+j)i},E_{(n+k)l}-E_{(n+l)k})$&$=$&$0, \quad  j,k,l\neq i,\ k\neq l;$\\
				$\theta(E_{i(n+j)}-E_{j(n+i)},E_{(n+k)l}-E_{(n+l)k})$&$=$&$0, \quad  j,k,l\neq i,\ k,l\neq j,\ k\neq l;$\\
				$\theta(E_{i(n+j)}-E_{j(n+i)},E_{jk}+E_{(n+k)(n+j)})$&$=$&$0, \quad j,k\neq i, \ j\neq k;$ \\
				$\theta(E_{i(n+k)}-E_{k(n+i)},E_{jk}+E_{(n+k)(n+j)})$&$=$&$\theta(E_{ii}+E_{(n+i)(n+i)},E_{i(n+j)}-E_{j(n+i)}), \quad j,k\neq i;$ \\
				$\theta(E_{(n+i)k}-E_{(n+k)i},E_{jk}+E_{(n+k)(n+j)})$&$=$&$0 \quad j,k\neq i, \ j\neq k;$\\
				$\theta(E_{(n+i)j}-E_{(n+j)i},E_{jk}+E_{(n+k)(n+j)})$&$=$&$\theta(E_{ii}+E_{(n+i)(n+i)},E_{(n+i)k}-E_{(n+k)i}), \quad j,k\neq i;$ \\
				$\theta(E_{i(n+j)}-E_{j(n+i)},E_{(n+j)k}-E_{(n+k)j})$&$=$&$\theta(E_{ii}+E_{(n+i)(n+i)},E_{ik}+E_{(n+k)(n+i)}), \quad  j,k\neq i.$
			\end{longtable}
			
			Finally, take the idempotents \begin{center}
				$a_{ij}=E_{ii}+E_{(n+i)(n+i)}+E_{ij}+E_{(n+j)(n+i)}+E_{i(n+j)}-E_{j(n+i)}$, 	\end{center}whose Peirce decomposition is
			\begin{longtable}{lclll}
				$\A_{1}^{a_{ij}}$&$=$& \multicolumn{3}{l}{$\mathbb{C} a_{ij};$}\\
				$\A_{0}^{a_{ij}}$&$=$&$\lsp\{$&
				$E_{ij}+E_{(n+j)(n+i)}+E_{i(n+j)}-E_{j(n+i)}-E_{jj}-E_{(n+j)(n+j)},E_{kl}+E_{(n+l)(n+k)},$&\\ 
				&&& $  E_{k(n+l)}-E_{l(n+k)},E_{(n+k)l}-E_{(n+l)k},E_{ik}+E_{(n+k)(n+i)}-E_{jk}-E_{(n+k)(n+j)},$&\\
				&&& $  E_{i(n+k)}-E_{k(n+i)}+E_{kj}+E_{(n+j)(n+k)},E_{ik}+E_{(n+k)(n+i)}-E_{(n+j)k}+E_{(n+k)j},$&\\ 
				&&& $ E_{i(n+k)}-E_{k(n+i)}-E_{j(n+k)}+E_{k(n+j)}\mid k,l\neq i,j$&$\};$\\
				$\A_{1/2}^{a_{ij}}$&$=$&$\lsp\{$&$E_{ii}+E_{(n+i)(n+i)}-E_{jj}-E_{(n+j)(n+j)}-E_{ji}-E_{(n+i)(n+j)},E_{ji}+E_{(n+i)(n+j)}$&\\ 
				&&& $  -E_{(n+j)i}+E_{(n+i)j},E_{ij}+E_{(n+j)(n+i)},E_{i(n+j)}-E_{j(n+i)},E_{ik}+E_{(n+k)(n+i)},$&\\ 
				&&& $  E_{i(n+k)}-E_{k(n+i)},E_{jk}+E_{(n+k)(n+j)}-E_{(n+j)k}+E_{(n+k)j}-E_{(n+i)k}+E_{(n+k)i},$&\\ 
				&&& $  E_{kj}+E_{(n+j)(n+k)}+E_{ki}+E_{(n+i)(n+k)}-E_{j(n+k)}+E_{k(n+j)}\mid k\neq i,j$&$\}.$
			\end{longtable}
			\noindent From condition~\eqref{axcond2}, we obtain
			\begin{longtable}{lcl}
				$\theta(E_{ij}+E_{(n+j)(n+i)},E_{i(n+j)}-E_{j(n+i)})$ &$=$&$0;$\\
				$\theta(E_{ij}+E_{(n+j)(n+i)},E_{(n+i)j}-E_{(n+j)i})$ &$=$&$0;$\\
				$\theta(E_{i(n+j)}-E_{j(n+i)},E_{(n+i)j}-E_{(n+j)i})$ &$=$&$0;$
			\end{longtable}
			\noindent and after that, we can apply condition~\eqref{axcond1} to get
			\begin{longtable}{lcl}
				$\theta(E_{i(n+j)}-E_{j(n+i)},E_{i(n+j)}-E_{j(n+i)})$ & $=$&$0;$\\
				$\theta(E_{(n+i)j}-E_{(n+j)i},E_{(n+i)j}-E_{(n+j)i})$ &$=$&$0.$
			\end{longtable}
			
			In summary, we have that $\theta=\delta f$ for 
			\begin{longtable}{rcl}
				$f\colon \A $& $\to$&$ {\rm V};$ \\
				$E_{ij}+E_{(n+j)(n+i)}$ & $\longmapsto$&$ \begin{cases} 2\theta(E_{ii}+E_{(n+i)(n+i)},E_{ij}+E_{(n+j)(n+i)}),\quad \textup{if } i\neq j; \\ 0,\quad \textup{if } i=j;\end{cases}$ \\
				$E_{i(n+j)}-E_{j(n+i)}$ & $\longmapsto$& $2\theta(E_{ii}+E_{(n+i)(n+i)},E_{i(n+j)}-E_{j(n+i)});$ \\
				$E_{(n+i)j}-E_{(n+j)i}$ & $\longmapsto$ &$ 2\theta(E_{ii}+E_{(n+i)(n+i)},E_{(n+i)j}-E_{(n+j)i}),$ \\
			\end{longtable}
			\noindent and therefore $[\theta]=0$.

			\item Type $\mathfrak{D}$.
			
			Let $\A$ be an algebra of type $\mathfrak{D}$, ${\rm V}$ a vector space over $\mathbb{C}$ and $\theta\colon \A\times \A \to {\rm V}$ a bilinear map. Consider  the idempotents $a_i=\frac{1}{2}(Ie_i+e_n)$ for $i=1,\dots,n-1$, where $I$ stands for the imaginary unit. The corresponding Peirce decompositions are
			\begin{longtable}{lcl}
				$\A_{1}^{a_i}$&$=$& $\mathbb{C} a_i;$\\
				$\A_{0}^{a_i}$&$=$&$\mathbb{C}(Ie_i-e_n);$\\
				$\A_{1/2}^{a_i}$&$=$&$\lsp\{e_j \mid j\neq i,n\}.$
			\end{longtable}
			\noindent Assume, without loss of generality, that $\theta(a_i,a_i)=0$ for all $i=1,\dots,n-1$; it follows that \[\theta(e_i,e_n)=I\theta(e_n,e_n)\] for $i\neq n$. In order that the idempotents  $a_i$, are $\mathcal{J}(\frac{1}{2})$-axes in $\A_{\theta}$, it must hold \[\theta(e_i,e_i)=-\theta(e_n,e_n)\] for $i\neq n$, by condition~\eqref{axcond1}, and that \[\theta(e_i,e_j)=0\] for $i,j\neq n, i\neq j$, by condition~\eqref{axcond2}. Then, we have that $\theta=\delta f$ for 
			\begin{longtable}{rcl}
				$f\colon \A $&$ \to$&$ {\rm V}$ \\
				$e_i$ & $\mapsto$&$ \begin{cases} I\theta(e_n,e_n),\quad \textup{if } i\neq n; \\ \theta(e_n,e_n),\quad \textup{if } i=n,\end{cases}$
			\end{longtable}
			\noindent and therefore $[\theta]=0$.
			
			\item Type $\mathfrak{E}$.
			
			Let us recall some notions regarding the octonions and the Jordan algebra $\textit{Herm}_3(\mathbb{O}_{\mathbb{C}})$, also known as Albert algebra, as well as establish some notation. For $x\in \mathbb{O}_{\mathbb{C}}$, write $x=x_0+\sum_{q=1}^7 I_q x_q$, with $x_q\in\mathbb{C}$. The multiplication table of the $I_q$'s is given by 
				$I_qI_r=-\delta_{qr}+\epsilon_{qrs}I_s$ where $\delta_{qr}$ is Kronecker's delta 
				and $\epsilon_{qrs}$ is the totally antisymmetric tensor with value $1$ when 
				$pqr=123,145,176,246,257,347,365$. The so-called Cayley involution of $\mathbb{O}_\mathbb{C}$ is 
				given by $x^*=(x_0+\sum_{q=1}^7 x_q I_q)^*=x_0-\sum_{q=1}^7 x_q I_q$. Also, take $q,r\in\{1,\dots,7\}$, $q\neq r$. If $I_qI_r=I_s$, for some $s\in\{1,\dots,7\}$, write $q\cdot r=s$, $s/r=q$ and $q\backslash s=r$; if $I_qI_r=-I_s$, write $q\cdot r=-s$, $s/r=-q$ and $q\backslash s=-r$.
			
			Let $\A$ be the $27$-dimensional algebra $\textit{Herm}_3(\mathbb{O}_{\mathbb{C}})$, given by all the matrices 
				\[\begin{pmatrix}\a & x & y\\x^* & \b & z\\ y^* & z^* &\gamma\end{pmatrix},
				\qquad \a,\b,\gamma\in\mathbb{C},\quad x,y,z\in\mathbb{O}_{\mathbb{C}}.\] The canonical basis of this algebra is\begin{center} $\{E_{ii}^0\}_{i=1}^3\cup \{E_{ij}^0+E_{ji}^0,E_{ij}^q-E_{ji}^q\}_{1\leq i<j\leq 3,
						q=1,\dots,7}$,\end{center} where:
				\begin{enumerate}
					\item  the unique non-zero entry of $E_{ii}^0$ is $1$
					in the $i$th row and column, $i=1,2,3$;
					\item  the non-zero entries of $E_{ij}^0+E_{ji}^0$ are $1$ in the $i$th row and $j$th column and 
					$1$ in the $j$th row and $i$th column, for $i,j\in\{1,2,3\}$ with $i<j$;
					\item   the non-zero entries of $E_{ij}^q-E_{ji}^q$ are $I_q$ in the $i$th row and $j$th column 
					and  $-I_q$ in the $j$th row and $i$th column, for $i,j\in\{1,2,3\}$ with $i<j$ and $q=1,\dots,7$.
				\end{enumerate}
				Finally, let us write $xE_{ij}=\sum_{q=0}^7 x_qE_{ij}^q$ and $E_{ij}^{-q}=-E_{ij}^q$ for $q=1,\dots,7$.

			The idempotents $a_i=E_{ii}^0$ have Peirce decomposition
			\begin{longtable}{lcl}
				$\A_{1}^{a_i}$&$=$& $\mathbb{C} a_i;$\\
				$\A_{0}^{a_i}$&$=$&$\lsp\{E_{jj}^0,E_{jk}^0+E_{kj}^0,E_{jk}^q-E_{kj}^q\mid j,k\neq i,\ j\neq k,\ q=1,\dots,7\};$\\
				$\A_{1/2}^{a_i}$&$=$&$\lsp\{E_{ij}^0+E_{ji}^0,E_{ij}^q-E_{ji}^q\mid j\neq i,\ q=1,\dots,7\}.$
			\end{longtable}
			\noindent Consider also a complex vector space ${\rm V}$ a bilinear map $\theta\colon \A\times \A \to {\rm V}$. Without loss of generality (see equation~\eqref{eq:thid0}), we will assume that $\theta(a_i,a_i)=0$. Applying condition~\eqref{axcond1}, we know that we must impose
			\begin{longtable}{lcl}
				$\theta(E_{ii}^0,E_{jj}^0)$&$=$&$0;$\\
				$\theta(E_{ii}^0,E_{jk}^0+E_{kj}^0)$&$=$&$0;$\\
				$\theta(E_{ii}^0,E_{jk}^q-E_{kj}^q)$&$=$&$0,$\\
			\end{longtable}
			\noindent for $j,k\neq i$, $j\neq k$ and $q=1,\dots,7$, so that the idempotents $a_i$ are semisimple in $\A_{\theta}$;  from condition~\eqref{axcond2}, we obtain that each $a_i$ is a $\mathcal{J}(\frac{1}{2})$-axis if and only if
			\begin{longtable}{lcl}
				$\theta(E_{ij}^0+E_{ji}^0,E_{jk}^0+E_{kj}^0)$&$=$&$\theta(E_{ii}^0,E_{ik}^0+E_{ki}^0),$ \\
				$\theta(E_{ij}^0+E_{ji}^0,E_{jk}^q-E_{kj}^q)$&$=$&$\theta(E_{ii}^0,E_{ik}^q-E_{ki}^q),$ \\
				$\theta(E_{ij}^q-E_{ji}^q,E_{jk}^q-E_{kj}^q)$&$=$&$-\theta(E_{ii}^0,E_{ik}^0+E_{ki}^0),$ \\
				$\theta(E_{ij}^q-E_{ji}^q,E_{jk}^r-E_{kj}^r)$&$=$&$\theta(E_{ii}^0,E_{ik}^{q\cdot r}-E_{ki}^{q\cdot r}),$
			\end{longtable}
			\noindent for $j,k\neq i$, $j\neq k$ and $q,s\in \{1,\dots,7\}$, $q\neq r$.
			
			Additionally, from the idempotents $a_{ij}^0=\frac{1}{2}(E_{ii}^0+E_{jj}^0+E_{ij}^0+E_{ji}^0)$, with Peirce decomposition
			\begin{longtable}{lclll}
				$\A_{1}^{a_{ij}^0}$&$=$& \multicolumn{3}{l}{$\mathbb{C} a_{ij}^0;$}\\
				$\A_{0}^{a_{ij}^0}$&$=$&$\lsp\{$ &$  E_{kk}^0,E_{ii}^0+E_{jj}^0-E_{ij}^0-E_{ji}^0,E_{ik}^0+E_{ki}^0-E_{jk}^0-E_{kj}^0,$\\ 
				&&&$E_{ik}^q-E_{ki}^q-E_{jk}^q+E_{kj}^q\mid j,k\neq i,\ j\neq k,\ q=1,\dots,7$&$\};$\\
				$\A_{1/2}^{a_{ij}^0}$&$=$&$\lsp\{$&$E_{ii}^0-E_{jj}^0,E_{ij}^q-E_{ji}^q,E_{ik}^0+E_{ki}^0+E_{jk}^0+E_{kj}^0,$\\ &&& 
				\multicolumn{1}{r}{$E_{ik}^q-E_{ki}^q+E_{jk}^q-E_{kj}^q\mid j\neq i,\ q=1,\dots,7$}&$\},$
			\end{longtable}
			\noindent and from conditions~\eqref{axcond1} and~\eqref{axcond2}, respectively, we obtain the next necessary conditions in order that that the idempotents $a_{ij}^0+\theta(a_{ij}^0,a_{ij}^0)$ are $\mathcal{J}(\frac{1}{2})$-axes in $\A_{\theta}$: 
			\begin{longtable}{lcl}
				$\theta(E_{ij}^0+E_{ji}^0,E_{ij}^0+E_{ji}^0)$&$=$&$0,$\\
				$\theta(E_{ij}^0+E_{ji}^0,E_{ij}^q-E_{ji}^q)$&$=$&$0,$
			\end{longtable}
			\noindent for $q=1,\dots, 7$.
			
			Finally, from the idempotents $a_{ij}^q=\frac{1}{2}(E_{ii}^0+E_{jj}^0+E_{ij}^q-E_{ji}^q)$, for $q=1,\dots,7$, with Peirce decomposition
			\begin{longtable}{lclll}
				$\A_{1}^{a_{ij}^q}$&$=$& $\mathbb{C} a_{ij}^q;$\\
				$\A_{0}^{a_{ij}^q}$&$=$&$\lsp\{$&$E_{kk}^0,E_{ii}^0+E_{jj}^0-E_{ij}^q+E_{ji}^q,E_{ik}^0+E_{ki}^0+E_{jk}^q-E_{kj}^q,E_{ik}^q-E_{ki}^q-E_{jk}^0-E_{kj}^0,$\\ 
				&&& \multicolumn{1}{r}{$E_{ik}^r-E_{ki}^r+E_{jk}^{r/q}-E_{kj}^{r/q} \mid i,j\neq k,\ r=1,\dots,7, \ r\neq q$}&$\};$\\
				
				$\A_{1/2}^{a_{ij}^q}$&$=$&$\lsp\{$&$E_{ii}^0-E_{jj}^0,E_{ij}^0+E_{ji}^0,E_{ij}^r-E_{ji}^r,E_{ik}^0+E_{ki}^0-E_{jk}^q+E_{kj}^q,E_{ik}^q-E_{ki}^q+E_{jk}^0+E_{kj}^0,$\\ 
				&&&\multicolumn{1}{r}{$E_{ik}^r-E_{ki}^r-E_{jk}^{r/q}+E_{kj}^{r/q}\mid i,j\neq k,\ r=1,\dots,7, \ r\neq q$}&$\}.$
			\end{longtable}
			\noindent and, again, from conditions~\eqref{axcond1} and~\eqref{axcond2}, respectively, it follows that \begin{longtable}{lcl}
				$\theta(E_{ij}^q-E_{ji}^q,E_{ij}^q-E_{ji}^q)$&$=$&$0,$\\			$\theta(E_{ij}^q-E_{ji}^q,E_{ij}^r-E_{ji}^r)$&$=$&$0,$
			\end{longtable}
			\noindent for $r=1,\dots, 7$, $r\neq q$, are also necessary conditions so that $a_{ij}^q+\theta(a_{ij}^q,a_{ij}^q)$ are $\mathcal{J}(\frac{1}{2})$-axes in  $\A_{\theta}$.
			
			Then, we have that $\theta=\delta f$ for 
			\begin{longtable}{rcl}
				$f\colon \A$ & $\to$&$ {\rm V}$ \\
				$E_{ij}^0+E_{ji}^0$ & $\longmapsto$&$ \begin{cases} 2\theta(E_{ii}^0,E_{ij}^0+E_{ji}^0),\quad \textup{if } i\neq j; \\ 0,\quad \textup{if } i=j;\end{cases}$ \\
				$E_{ij}^q-E_{ji}^q $& $\longmapsto$&$ 2\theta(E_{ii}^0,E_{ij}^q-E_{ji}^q),$
			\end{longtable}
			\noindent and therefore $[\theta]=0$.
		\end{itemize}
		
	\end{proof}

	We devote the rest of this section to study $\mathcal{J}(\frac{1}{2})$-axial central extensions of different  Jordan algebras generated by idempotents. 
	
	Denote by $\mathcal{S}_n$, $\mathcal{J}_n$ and $\mathcal{T}_n$, respectively, the following Jordan algebras of dimension $n$: 
	
	
	\begin{longtable}{llllll}
		$\mathcal{S}_n$ & $e_ie_i=e_i $ &&&& $\quad i=1,\dots,n$;   \\
		$\mathcal{J}_n$ & $ee=e$ & $en_i=\frac{1}{2}n_i$ &&& $ \quad i=1,\dots,n-1$;  \\
		$\mathcal{T}_n$ & $ee=e$ & $en_{1}=n_{1}$ & $n_{2}n_2=n_1$ & $en_i=\frac{1}{2}n_i$ & $ \quad i=2,\dots,n-1$,
	\end{longtable}
	\noindent where the non-displayed products are assumed to be zero. 
	
	\begin{prop}\label{p:sn}
		There do not exist $\mathcal{J}(\frac{1}{2})$-axial non-split central extensions of $(\mathcal{S}_n,{\rm X})$ for any set ${\rm X}$ of generating axes of $\mathcal{S}_n$.
	\end{prop}
	
	\begin{proof}
		Fix $n\in \mathbb{Z}_+$, and let ${\rm X}$ be a set of axes 
		generating $\mathcal{S}_n$. Note that $\mathcal{S}_n$ is just the linear span of $\rm X$. Any $a\in {\rm X}$ satisfies $a=\sum_{i=1}^n \alpha_i^a e_i$, with $\alpha_i^a\in\{0,1\}$. Straightforward calculations show that 
		\begin{longtable}{lcl}
			$(\mathcal{S}_n)_1^a$&$=$&$\mispan\{ e_i\mid \alpha_i=1\};$ \\
			$(\mathcal{S}_n)_0^a$&$=$&$\mispan\{ e_i\mid \alpha_i=0\};$\\
			$(\mathcal{S}_n)_{1/2}^a$&$=$&$0.$
		\end{longtable}
		
		
		Let $(\mathcal{S}_n)_{\theta}$ be a $\mathcal{J}(\frac{1}{2})$-axial central extension of $(\mathcal{S}_n,{\rm X})$ with respect to a vector space $\rm V$, and fix $i\in\{1,\dots,n\}$. As ${\rm X}$ spans $\mathcal{S}_n$, there exists an axis $a_i\in X$ such that $\alpha^{a_i}_i=1$; by condition~\eqref{axcond2}, $\theta(e_i,e_j)=0$ for all $j=1,\dots,n$, $j\neq i$. It follows that $\theta(e_i,e_j)=0$ for all $i,j=1,\dots,n$, $i\neq j$. Necessarily, we have that $\theta=\delta f $ for 
		\begin{longtable}{rcl}
			$f\colon \mathcal{S}_n$ & $\to$ &$ {\rm V}$ \\
			$e_i$ & $\longmapsto$ & $\theta(e_i,e_i),$
		\end{longtable}
		\noindent and thus $[\theta]=0$.
	\end{proof}
	
	\begin{prop}\label{p:jn}
		Every commutative central extension of $\mathcal{J}_n$ is a Jordan algebra. 
	\end{prop} 
	
	\begin{proof}
		Fix $n\in \mathbb{Z}_+$, and let $(\mathcal{J}_n)_{\theta}$ be a commutative central extension of $\mathcal{J}_n$. Take two arbitrary elements $x,y\in \mathcal{J}_n$, and write $x=\alpha e + v$, $y=\beta e + w$ for some $v,w\in\lsp\{n_i\mid i=1,\dots,n-1\}$. One can check that
			\begin{longtable}{lcl}
				$xy$&$=$& $\alpha\beta e+ \frac{1}{2}\beta v + \frac{1}{2} \a w;$ \\
				$x^2$&$=$&$\alpha^2e+\a v;$ \\
				$yx^2$&$=$&$\alpha^2\b e + \frac{1}{2}\a\beta v + \frac{1}{2} \a^2 w,$
			\end{longtable}
			\noindent and that
			\[\theta(xy,x^2)=\theta(x,yx^2)=\alpha^3\b \theta(e,e) + \frac{3}{2} \a^2\b \theta(e,v) + \frac{1}{2} \a^3 \theta(e,w) + \frac{1}{2} \a\b \theta(v,v) + \frac{1}{2} \a^2 \theta(v,w).\]
			Then, $\theta$ is a $2$-cocycle in the variety of Jordan algebras, and $(\mathcal{J}_n)_{\theta}$ is Jordan.
	\end{proof}
	
	
	In view of Proposition~\ref{p:jn}, we highlight the major difference between the cohomology of $\mathcal{J}_n$ with trivial coefficients and the cohomology with coefficients in itself (see~\cite[Theorem 10]{gkp}).
	
	
	Proposition~\ref{p:jn} has the following immediate corollary.
	
	\begin{cor}\label{c:jn}
		Every $\mathcal{J}(\frac{1}{2})$-axial central extension of $(\mathcal{J}_n,{\rm X})$, for any set $\rm X$ of generating axes of $\mathcal{J}_n$, is a Jordan algebra.
	\end{cor}	
	
	\begin{prop}\label{p:tn}
		Every $\mathcal{J}(\frac{1}{2})$-axial central extension of $(\mathcal{T}_n,{\rm X})$, for any set $\rm X$ of generating axes of $\mathcal{T}_n$, is a Jordan algebra.
	\end{prop}
	
	\begin{proof}
		Fix $n\in \mathbb{Z}_+$, and let ${\rm X}$ be a set of axes 
		generating $\mathcal{T}_n$. Take $a\in {\rm X}$; it is of the form $a=e-\alpha_2^2n_1+\sum_{i=2}^{n-1} \alpha_i n_i$. Some calculations lead to 
		\begin{longtable}{lcl}
			$(\mathcal{T}_n)_1^a$&$=$&$\mispan\{a,n_1\};$ \\
			$(\mathcal{T}_n)_0^a$&$=$&$0;$ \\
			$(\mathcal{T}_n)_{1/2}^a$&$=$&$\mispan\{2\alpha_2 n_1-n_2,n_i\mid i=3,\dots,n-1\}.$
		\end{longtable}
		\noindent Assume now that $\theta\in\Za(\mathcal{T}_n,{\rm V};{\rm X})$ for some vector space $\rm V$. Applying condition~\eqref{axcond2}, we have that \[\theta(n_1,n_i)=\theta(a,0)=0\] for $i=1$ and $i=3,\dots,n-1$, and \[\theta(n_1,2\alpha_2 n_1-n_2)=\theta(a,0)=0,\] which implies $\theta(n_1,n_2)=0$. Now, it is just a matter of routine to check that $\theta(xy,x^2)=\theta(x,yx^2)$ for any two arbitrary $x,y\in\mathcal{T}_n$, as it was made in Proposition~\ref{p:jn}.
	\end{proof}
	
	An easy generalisation of Corollary~\ref{c:jn} and Proposition~\ref{p:tn} is the following.
	
	\begin{cor}\label{c:ds}
		Let $n,m$ be two positive integers.
		Every $\mathcal{J}(\frac{1}{2})$-axial central extension of $(\mathcal{J}_n\oplus \mathcal{J}_m,{\rm X})$ or $(\mathcal{T}_n\oplus \mathcal{T}_m,{\rm X})$, for any sets $\rm X$ of generating axes, is a Jordan algebra.
	\end{cor}
	
	\begin{proof}
		Let $(\mathcal{R}=\mathcal{R}_n\oplus \mathcal{R}_m,{\rm X})$ stand for $(\mathcal{J}_n\oplus \mathcal{J}_m,{\rm X})$ or $(\mathcal{T}_n\oplus \mathcal{T}_m,{\rm X})$. Take $a\in {\rm X}$,  $a=a^n + a^m$ where $a^n$, $a^m$ are idempotents 
		of $\mathcal{R}_n$ and $\mathcal{R}_m$, respectively. We claim that $\mathcal{R}_0^a=(\mathcal{R}_n)_0^{a^n}\oplus (\mathcal{R}_m)_0^{a^m}=0$. For $\mathcal{T}_n\oplus \mathcal{T}_m$, it follows from the proof of Proposition~\ref{p:tn}; for $\mathcal{J}_n\oplus \mathcal{J}_m$, it suffices to observe that, for any  idempotent $a\in \mathcal{J}_n$, $a$ is of the form $a=e+\sum_{i=1}^{n-1}\alpha_i n_i$ and satisfies $(\mathcal{J}_n)_0^a=0$.  Let $\theta\in\Za(\mathcal{R},{\rm V};{\rm X})$ for some vector space $\rm V$. By condition~\eqref{axcond2}, $\theta(\mathcal{R}_n,\mathcal{R}_m)=0$. The rest of the proof follows from Corollary~\ref{c:jn} and Proposition~\ref{p:tn}.
	\end{proof}
	
	We include now some straightforward considerations which will be useful for our last result.
	
	\begin{rem}\label{r:easy}
			\hfill
			\begin{enumerate}
				\item Any axial algebra is perfect.
				\item An algebra $\A=\A_1\oplus\A_2$ is  $(\mathcal{F},\star)$-axial if and only if $\A_1$ and $\A_2$ are both $(\mathcal{F},\star)$-axial algebras.
			\end{enumerate}
	\end{rem}
	%
	%
	
	\begin{lm}\label{l:nilp}
		Let $\A$ be a Jordan algebra. 
		If there exists a pair $\{e,x\}\subset \A$ such that $e$ is idempotent, $x\neq 0$ belongs to the $1$-eigenspace of $e$, and there exists a linear complement $L$ of $\mispan\{e,x\}$ in $\A$ which is an ideal and satisfies $x^2\in L$, then $\A$ is not axial.
	\end{lm}
	
	
	\begin{proof}
		Let $\{e_1,\dots,e_{r}\}$ be a basis of $L$, and suppose that $a=\lambda e+\mu x +\sum_{i=1}^{r}\alpha_i e_i$ is an axis. Then $a^2=a$ and, consequently, $\lambda^2=\lambda$ and $2\lambda\mu=\mu$. This yields $\mu=0$. It follows that there cannot exist a set of axes generating $\A$. 
	\end{proof}
	
	\begin{lm}\label{l:suf}
		Let $\A$ be a Jordan algebra, 
		and let $\A=\A_s \dot{+} \A_n$ be its decomposition into semisimple and nilpotent parts. If there exist a set of generators of $\A_n$, $\{x_1,\dots,x_r\}$, such that $x_i$ belongs to the $1/2$-eigenspace of an axis $a(x_i)$ of $\A_s$ and that $x_i^2=0$ for all $i=1,\dots,r$, then $\A$ is an axial algebra.
	\end{lm}
	
	\begin{proof}
		The semisimple component $\A_s$ is trivially axial for being a direct sum of simple Jordan algebras. Take a set $\rm X_s$ of generating axes such that $a(x_i)\in \rm X_s$ for all $i=1,\dots,r$. Then, ${\rm X_s} \cup \{a(x_i)+x_i\}_{i=1}^{r}$ is a generating set of axes for $\A$. 
	\end{proof}
	
	\begin{thm}\label{th:j4}
		Let $\A$ be an axial Jordan algebra of dimension $n\leq 4$ over an algebraically closed field $F$. 
		Then, every $\mathcal{J}(\frac{1}{2})$-axial non-split central extension of $(\A,{\rm X})$, for any set $\rm X$ of generating axes of $\A$, is a Jordan algebra.
	\end{thm}
	
	\begin{proof} 
		Employing Remark~\ref{r:easy} and Lemma~\ref{l:nilp}, we select from the lists of $n$-dimensional Jordan algebras in~\cite{KS,martin}, for $n\leq 4$, those who are axial. It is obvious that there are not non-split $\mathcal{J}(\frac{1}{2})$-axial central extensions of the axial Jordan algebra of dimension $1$. Also, by~\cite{hrs2,gs}, we know that the theorem holds true for every axial Jordan algebra generated by two or three primitive axes; then we use the help of Lemma~\ref{l:suf} to identify them. Indeed, in the next Table~\ref{t:j} we display such algebras $\mathfrak{J}$ together with a set of primitive generating axes ${\rm X}$. We follow the notation of~\cite{martin}.
		\newpage
		
		\begin{longtable}{|c|c|}
			\hline
			$\A$  & ${\rm X}$  \\
			\hline \hline $\mathfrak{F}_1\oplus\mathfrak{F}_1$ & $\{e_1,e_2\}$ \\
			\hline $\mathfrak{B}_2$ & $\{e_1,e_1+n_1\}$ \\
			\hline $\mathfrak{F}_1\oplus\mathfrak{F}_1\oplus\mathfrak{F}_1$ & $\{e_1,e_2,e_3\}$ \\
			\hline $\mathfrak{B}_2\oplus\mathfrak{F}_1$ & $\{e_1,e_1+n_1,e_2\}$ \\
			\hline $\mathfrak{T}_5$ & $\{e_1,1/2(e_1+e_2+e_3)\}$ \\
			\hline $\mathfrak{T}_7$ & $\{e_1,e_1+n_1,e_1+n_2\}$ \\
			\hline $\mathfrak{T}_8$ & $\{e_1,e_1+n_1+n_2\}$ \\
			\hline $\mathfrak{T}_{10}$ & $\{e_1+n_1,e_2+n_1\}$ \\
			\hline $\mathfrak{J}_{1}$  & $\{e_1,1/2(e_1+e_2+e_3),e_4\}$ \\ 
			\hline $\mathfrak{J}_{2}$  & $\{e_1+e_3,e_1+e_4,e_2\}$  \\
			\hline $\mathfrak{J}_{7}$  & $\{e_1+n_1,e_2+n_1,e_3\}$  \\
			\hline $\mathfrak{J}_{9}$  & $\{e_1,1/2(e_1+e_2+e_3),e_1+n_1\}$  \\
			\hline $\mathfrak{J}_{16}$  & $\{e_1+n_1,e_1+n_2,e_2\}$  \\
			\hline $\mathfrak{J}_{18}$  & $\{e_1+n_1,e_1+n_2,e_2\}$  \\
			\hline $\mathfrak{J}_{23}$  & $\{e_1,e_1+n_1+n_2,e_2\}$  \\
			\hline $\mathfrak{J}_{48}$  & $\{e_1,e_1+n_1+n_3,e_1+n_2\}$  \\
			\hline $\mathfrak{J}_{49}$  & $\{e_1,e_1+n_1+n_3,e_1+n_2\}$  \\
			\hline 
			
			\caption{Jordan algebras of dimension $n=2,3,4$ generated by two or three primitive axes.}
			\label{t:j}
		\end{longtable}

		It remains to consider the following algebras:
		
		\begin{longtable}{llllll}
			$\mathfrak{T}_9\equiv\mathcal{T}_3$ & $\mathfrak{J}_{3}\equiv \mathcal{S}_4$ & $\mathfrak{J}_{6}\equiv\mathcal{J}_2\oplus\mathcal{J}_1\oplus\mathcal{J}_1$ & $\mathfrak{J}_{12}\equiv\mathcal{J}_3\oplus\mathcal{J}_1$ & $\mathfrak{J}_{13}\equiv\mathcal{J}_2\oplus\mathcal{J}_2$ & $\mathfrak{J}_{24}\equiv\mathcal{T}_3\oplus \mathcal{T}_1$ \\
			$\mathfrak{J}_{25}$ & $\mathfrak{J}_{33}\equiv\mathcal{J}_4$ & $\mathfrak{J}_{53}$ & $\mathfrak{J}_{58}\equiv\mathcal{T}_4 $ & $\mathfrak{J}_{59}$. &
		\end{longtable}
		
		
		\noindent But Proposition~\ref{p:sn} and Corollary~\ref{c:ds} ensure that we only need to study the algebras $\mathfrak{J}_{25}$, $\mathfrak{J}_{53}$ and $\mathfrak{J}_{59}$. We include their tables of products for the sake of completeness.
		\begin{longtable}{lllllll}
			$\mathfrak{J}_{25}$  & $e_1e_1=e_1$ & $e_1n_1= \frac{1}{2} n_1$ &  $e_1n_2=n_2$ & $e_2e_2=e_2$ & $e_2n_1= \frac{1}{2} n_1$ & $n_1n_1=n_2$  \\
			$\mathfrak{J}_{53}$  & $ee=e$ &  $en_1= \frac{1}{2} n_1$ & $en_2=n_2$ & $n_1n_1=n_2+n_3$& & \\
			$\mathfrak{J}_{59}$  & $ee=e$ &  $en_1= n_1$ & $en_2=\frac{1}{2} n_2$ & $en_3=\frac{1}{2} n_3$ & $n_2n_3=n_1$ & $n_3n_3=n_1$
		\end{longtable}

		Let us begin studying $\mathfrak{J}_{25}$. We find three different types of idempotents: the unity $e=e_1+e_2$ and
		\begin{longtable}{lcl}
			$a(\a)$&$=$&$e_1+\a n_1 - \a^2n_2,$\\
			$b(\b)$&$=$&$ e_2 + \b n_1 + \b^2n_2.$
		\end{longtable}
		\noindent The respective eigenspace decompositions of $a(\a)$ and $b(\b)$ are:
		\begin{longtable}{lcl}
			$(\mathfrak{J}_{25})_1^{a(\a)}$&$=$&$\mispan\{a(\a),n_2\};$ \\
			$(\mathfrak{J}_{25})_0^{a(\a)}$&$=$&$F (e_2-\a n_1 + \a^2 n_2);$ \\
			$(\mathfrak{J}_{25})_{1/2}^{a(\a)}$&$=$&$F (n_1-2\a n_2);$\\
			
			$(\mathfrak{J}_{25})_1^{b(\b)}$&$=$&$Fb(\b);$ \\
			$(\mathfrak{J}_{25})_0^{b(\b)}$&$=$&$\mispan \{e_1-\b n_1, n_2\};$ \\
			$(\mathfrak{J}_{25})_{1/2}^{b(\b)}$&$=$&$F (n_1+2\b n_2).$
		\end{longtable}
		
		As usual, let $\rm X$ be a set of idempotents which generate $\mathfrak{J}_{25}$, and let $\theta\in\Za(\mathfrak{J}_{25},{\rm V};{\rm X})$ for some vector space $\rm V$. Suppose first that $e\in {\rm X}$. Condition~\eqref{axcond2} tells immediately that \[\theta(e_1,e_2)=\theta(e_2,n_2)=\theta(n_1,n_2)=\theta(n_2,n_2)=\theta(e,0)=0;\] also, that \[\theta(n_1,n_1)=\theta(e,n_2)=\theta(e_1,n_2)\] and \[\theta(e_1,n_1)=\theta(e,\frac{1}{2}n_1),\] from which \[\theta(e_1,n_1)=\theta(e_2,n_1).\]
		
		On the contrary, assume now that $e\notin {\rm X}$. It is not possible to generate $\mathfrak{J}_{25}$ only with axes of type $a$ or of type $b$, so we can take $a(\a), b(\b) \in {\rm X}$. Condition~\eqref{axcond1} ensures that $\theta(b(\b),n_2)=\theta(b(\b),e_1-\b n_1)=0$; on the other hand, according to condition~\eqref{axcond2}, $\theta(n_2,n_2)=\theta(n_2,n_1-2\a n_2)=\theta(a(\a),0)=0$, $\theta(n_1-2\a n_2,n_1-2\a n_2)=\theta(a(\a),n_2)$ and $\theta(e_1-\b n_1,n_1+2\b n_2)=2\theta(b(\b),\frac{1}{2}n_1+\b n_2)$. Putting this information together, again we obtain \[\theta(e_1,e_2)=\theta(e_2,n_2)=\theta(n_1,n_2)=\theta(n_2,n_2)=0,\] \[\theta(n_1,n_1)=\theta(e_1,n_2)\] and \[\theta(e_1,n_1)=\theta(e_2,n_1).\]
		
		It is now routine to apply these equalities to the expressions of $\theta(xy,x^2)$ and $\theta(x,yx^2)$ to check that they coincide for every $x,y\in \mathfrak{J}_{25}$. 
		
		\
		
		The two last cases will be dealt with similarly, although they are simpler. That is why we will omit the details in our exposition. 
		
		Let $a$ be an idempotent of $\mathfrak{J}_{53}$; then $a$ is of the form $a=e+\a n_1 - \a^2 n_2 +\a^2 n_3$. We have:
		\begin{longtable}{lcl}
			$(\mathfrak{J}_{53})_1^a$&$=$&$\mispan\{a,n_2\};$ \\
			$(\mathfrak{J}_{53})_0^a$&$=$&$F n_3;$ \\
			$(\mathfrak{J}_{53})_{1/2}^a$&$=$&$F (n_1-2\a n_2+2\a n_3).$
		\end{longtable}
		\noindent Let $\theta\in \Za(\mathfrak{J}_{53},{\rm V};{\rm X})$ for any set ${\rm X}$ of generating axes and for some vector space $\rm V$. Then, employing conditions~\eqref{axcond1} and~\eqref{axcond2} we obtain that $\theta=\delta f$ for the linear map
		\begin{longtable}{rcl}
			$f\colon \mathfrak{J}_{53}$ & $\to $&${\rm V}$ \\
			$e$ & $\longmapsto$&$ \theta(e,e),$ \\
			$n_1$ & $\longmapsto$&$ 2\theta(e,n_1),$ \\
			$n_2$ & $\longmapsto$&$ \theta(e,n_2),$ \\
			$n_3$ & $\longmapsto$&$ \theta(n_1,n_1)-\theta(e,n_2),$ 
		\end{longtable}
		\noindent and therefore the extension $(\mathfrak{J}_{53})_{\theta}$ is split.
		
		Finally, let $a\in \mathfrak{J}_{59}$ be an idempotent; it will satisfy $a=e-\b(2\a+\b)n_1 +\a n_2 + \b n_3$. Then:
		\begin{longtable}{lcl}
			$(\mathfrak{J}_{59})_1^a$&$=$&$\mispan\{a,n_1\};$ \\
			$(\mathfrak{J}_{59})_0^a$&$=$&$0;$ \\
			$(\mathfrak{J}_{59})_{1/2}^a$&$=$&$\mispan\{ 2\b n_1-n_2, 2(\a+\b)n_1- n_3\}.$
		\end{longtable}
		\noindent Let $\theta\in \Za(\mathfrak{J}_{59},{\rm V};{\rm X})$ for any set ${\rm X}$ of generating axes and for some vector space $\rm V$. Condition~\eqref{axcond2} allows us to verify that $\theta(xy,x^2)=\theta(x,yx^2)$ for any two arbitrary $x,y\in\mathfrak{J}_{59}$, so the extension $(\mathfrak{J}_{59})_{\theta}$ is Jordan.
		
	\end{proof}
	
	Theorems~\ref{th:j} and~\ref{th:j4},  together with Propositions~\ref{p:sn} and~\ref{p:tn} and Corollary~\ref{c:jn}, lead us to make the following conjecture.

	\begin{conj}
		Let $F$ be a field of characteristic different from $2$. Every $\mathcal{J}(\frac{1}{2})$-axial central extension of a Jordan algebra over $F$ generated by idempotents, with respect to any set of such generating idempotents, is a Jordan algebra.
	\end{conj}

	\section{Acknowledgements}
	
	The first part of the work was supported by the Austrian Science Foundation FWF, grant P 33811-N, by Agencia Estatal de Investigación (Spain),
	grant PID2020-115155GB-I00 (European FEDER support included, UE) and by Xunta de Galicia, grant ED431C 2019/10 (European FEDER support included, UE); by the Junta de Andaluc\'{\i}a  through projects UMA18-FEDERJA-119  and FQM-336 and  by the Spanish Ministerio de Ciencia e Innovaci\'on   through project  PID 2019-104236GB-I00,  all of them with FEDER funds;
	FCT   UIDB/MAT/00212/2020 and UIDP/MAT/00212/2020.
	The second part of the work was supported by RSF  22-71-10001.

\end{document}